\documentclass[12pt]{amsart}
\usepackage{amsfonts,amscd,latexsym,amsmath,amssymb,cancel,enumerate}
\theoremstyle{plain}
\newtheorem{master}{Master}[section]
\newtheorem{prop}[master]{Proposition}
\newtheorem{thm}[master]{Theorem}

\newtheorem{lem}[master]{Lemma}

\newtheorem{claim}[master]{Claim}
\theoremstyle{definition}
\newtheorem{defin}[master]{Definition}
\newtheorem{observation}[master]{Observation}
\theoremstyle{remark}
\newtheorem{remark}[master]{Remark}
\numberwithin{equation}{section}
\newcommand{\Ur}{\mathbb{U}}
\newcommand{\Rea}{\mathbb{R}}
\newcommand{\Nat}{\mathbb{N}}
\newcommand{\Rat}{\mathbb{Q}}
\newcommand{\Age}{\mathcal{K}}
\newcommand{\Lng}{\mathrm{L}}
\begin{document}
\title{Universal and ultrahomogeneous Polish metric structures}
\author{Michal Doucha}
\address{Institute of Mathematics, Academy of Sciences, Prague, Czech republic}
\email{m.doucha@post.cz}
\thanks{The research of the author was partially supported by grant IAA100190902 of Grant Agency of the Academy of Sciences of the Czech Republic}
\begin{abstract}
We use Fra\" iss\' e theoretic methods to construct several universal and ultrahomogeneous Polish metric structures. Namely, universal and ultrahomogeneous Polish metric space equipped with countably many closed subsets of its powers, universal and ultrahomogeneous Polish metric space equipped with a closed subset of the product of itself and some fixed compact metric space, and universal and ultrahomogeneous Polish metric space equipped with an $L$-Lipschitz function, for an arbitrary positive $L$, to some fixed Polish metric space. These results are direct generalization of the classical result of P. Urysohn \cite{Ur}. Possible applications are discussed.
\end{abstract}
\maketitle
\section*{Introduction}
In 1927, P. S. Urysohn constructed a metric space $\Ur$ which is now called The Urysohn universal metric space (\cite{Ur}). It is a Polish metric space that is both universal and ultrahomogeneous for the class of all finite metric spaces. The universality means that every finite metric space can be isometrically embedded into $\Ur$ and the ultrahomogeneity means that any finite isometry $\phi: \{x_1,\ldots,x_n\}\subseteq \Ur\rightarrow \{y_1,\ldots,y_n\}\subseteq \Ur$ extends to an isometry $\bar{\phi}\supseteq \phi:\Ur\rightarrow \Ur$ on the whole space. These two properties imply that $\Ur$, in fact, contains an isometric copy of every separable metric space and that $\Ur$ is unique with these two properties up to isometry.

The aim of this paper is to enrich the Urysohn space with some additional structure so that this enriched Urysohn space is still universal and ultrahomogeneous for that specific (Polish) metric structure. The definition of Polish metric structures considered here is given at the end of this section. A related work has been done by W. Kubi\' s in \cite{Kub} (see also \cite{GK}).

Our initial motivation was to provide a general way of coding of such classes of Polish metric structures as standard Borel spaces. Let us say we are given some class of Polish metric structures and we would like to use methods of descriptive set theory to investigate (e.g. classify) this class. In order to use these methods we need to represent such a class as a Polish space or, it is sufficient, as a standard Borel space. Recall the definition of an Effros-Borel structure (see for example \cite{Ke}). Effors-Borel structure is an example of a standard Borel space that can serve in this direction. Let us illustrate it on examples.\\

\noindent {\bf Examples.}
\begin{itemize}
\item Since every Polish space $X$ is homeomorphic to a closed subset of $\Rea^\Nat$, the Effros-Borel space of $F(\Rea^\Nat)$ can be interpreted as a standard Borel space of all Polish spaces.
\item Recall the classical result of Banach and Mazur that every separable Banach space can be be embedded by a linear isometry into the separable Banach space $C([0,1])$, i.e. the Banach space of all real-continuous functions on $[0,1]$. Consider the following subset of the standard Borel space $F(C([0,1]))$, which can be checked to be Borel, $\mathrm{Subs}=\{X\in F(C([0,1])): X\text{ is a linear subspace}\}$. It is a standard Borel space of all separable Banach spaces (that has been used, for instance, by V. Ferenczi, A. Louveau and C. Rosendal in \cite{FLR} for a classification result of separable Banach spaces with the relation of linear isomorphism). There are a lot of Borel subsets of $\mathrm{Subs}$ that represent certain subclasses of separable Banach spaces (see \cite{Do} for examples).
\item Because of the properties of $\Ur$ the standard Borel space $F(\Ur)$ can serve as a coding of all Polish metric spaces. We remark that this approach was used by Gao and Kechris in \cite{GaKe} in their classification of Polish metric spaces up to isometry.

\end{itemize}
The Effros-Borel structure of $F(\mathcal{S})$, where $\mathcal{S}$ will be one of the structures we investigate here, should serve in a similar way. Let us state the main definitions of this chapter.
\begin{defin}[Polish metric structure]
Let $Z_1,\ldots,Z_k$ be a list of Polish metric spaces. A finite or countably infinite set $\mathcal{O}$ is called a signature if it consists of symbols for closed sets. Moreover, there is a function $a:\mathcal{O}\rightarrow ([0,\ldots,k]\times \Nat)^{<\omega}$; i.e. to each symbol from $\mathcal{O}$ it assigns a finite sequence of elements $(a,b)$ where $0\leq a\leq k$ and $b\in \Nat$. By $a_F(n,i)$, for $i\in\{1,2\}$, we denote the $i$-th coordinate of the $n$-th element of $a(F)$.

A Polish metric structure of signature $\mathcal{O}$ is a Polish metric space $(X,d)$ such that for every $F\in \mathcal{O}$ there is a closed set $F_X\subseteq Z_{a_F(1,1)}^{a_F(1,2)}\times\ldots \times Z_{a_F(|a(F)|,1)}^{a_F(|a(F)|,2)}$, where by $Z_0$ we denote $X$.
\end{defin}
\begin{defin}
Let $(X,d,\mathcal{O}_X)$ be a Polish metric space of some signature $\mathcal{O}$. We say that $(X,d,\mathcal{O}_X)$ is universal if for any Polish metric space $(Y,d,\mathcal{O}_Y)$ of the same signature $\mathcal{O}$ there is an isometric embedding $\phi:Y\hookrightarrow X$ that moreover reduces $F_Y$ into $F_X$ (for every $F\in \mathcal{O}$): i.e. for any $(y_1,\ldots,y_n)$, where $I\subset \{1,\ldots,n\}$ are the coordinates such that $y_i\in Y$ iff $i\in I$, we have $(y_1,\ldots,y_n)\in F_Y\Leftrightarrow (x_1,\ldots,x_n)\in F_X$, where $x_i=\phi(y_i)$ if $i\in I$ and $x_i=y_i$ otherwise.

We say that $(X,d,\mathcal{O}_X)$ is ultrahomogeneous if any isomorphism between two finite (metric) substructures $(F_1,d,\mathcal{O}_{F_1})$ and $(F_2,d,\mathcal{O}_{F_2})$ of $(X,d,\mathcal{O}_X)$ extends to an automorphism of the whole $(X,d,\mathcal{O}_X)$.
\end{defin}
Let us illustrate the universality and ultrahomogeneity on examples.\\

\noindent{\bf Examples.}
\begin{itemize}
\item If the signature $\mathcal{O}$ is empty then $(X,d,\mathcal{O}_X)$ is just the Urysohn universal metric space $\Ur$, i.e. space containing an isometric copy of every Polish (or just separable) metric space and with the property that every finite partial isometry extends to an isometry of the whole space.
\item Let us consider the case when the signature $\mathcal{O}_X$ contains a symbol for one closed subset $C$ of $X$. Then for any Polish metric space $(Y,d)$ equipped with some closed subset $D\subseteq Y$ there is an isometric embedding $\phi:Y\hookrightarrow X$ that maps $D$ into $C$, i.e. $\forall y\in Y (y\in D\Leftrightarrow \phi(y)\in C)$; in other words, $\phi(Y)\cap C=\phi(D)$. Moreover, for any two finite subspaces $F_1,F_2\subseteq X$ and an isometry $\phi:F_1\rightarrow F_2$ respecting the closed subset, i.e. $\phi(F_1\cap C)=F_2\cap C$, there is an extension to an isometry on the whole space $\phi\subseteq \bar{\phi}:X\rightarrow X$ that still respects the closed subset, i.e. $\bar{\phi}(C)=C$.
\item Let us consider the case when the signature $\mathcal{O}_X$ contains a symbol for a closed subset $C$ of $X\times Z$ where $Z$ is some fixed Polish metric space. Then for any Polish metric space $(Y,d)$ and a closed subset $D\subseteq Y\times Z$ there is an isometric embedding $\phi:Y\hookrightarrow X$ such that $\forall y\in Y\forall z\in Z ((y,z)\in D\Leftrightarrow (\phi(y),z)\in C)$. Moreover, for any two finite substructures $F_1,F_2$ and isometry $\phi$ between them respecting the structure, i.e. $\forall f\in F_1\forall z\in Z((f,z)\in C\Leftrightarrow (\phi(f),z)\in C)$, there is an extension $\bar{\phi}$ to the whole space still respecting the closed set $C$.
\item Let us consider the case when the signature $\mathcal{O}_X$ contains a symbol $f$ for a closed subset of $X\times Z$ and moreover $(X,d,\mathcal{O}_X)\models f\text{ is a graph of a continuous function}$, where $Z$ is some fixed Polish metric space. Then for any Polish metric space $(Y,d)$ equipped with a continuous function $g:Y\rightarrow Z$ (i.e. a Polish metric structure of that signature which also models that this closed set is in fact a graph of a continuous function) to that fixed space $Z$ there is an isometric embedding $\phi:Y\hookrightarrow X$ that maps the graph of $g$ into the graph of $f$; in other words, $\forall y\in Y(g(y)=f\circ \phi (y))$. Moreover, for any two finite metric substructures $F_1,F_2\subseteq X$ and an isometry $\phi$ between them that respects the continuous function, i.e. $\forall x\in F_1 (f(x)=f\circ \phi(x))$, there is an extension to the isometry on the whole space that still respects the continuous function.

\end{itemize}
In what follows we shall denote the Polish metric structures somewhat loosely. For instance the Polish metric structure with two closed sets would be denoted often as $(X,F_1,F_2)$ instead of $(X,d,F_X^1,F_X^2)$ where $F^1,F^2$ are two symbols for closed sets.
\begin{defin}[Almost universal and ultrahomogeneous structures]
Suppose now that the signature $\mathcal{O}$ on $(X,d)$ consists of countably many symbols for closed sets of the same type, e.g. countably many closed subsets of $X$ or countably many continuous functions (resp. graphs of them) from $X$ to some fixed metric spaces. In such a case we will usually not be able to maintain universality and ultrahomogeneity in the full strength. Let us have $\mathcal{O}$ enumerated as $\{O_n:n\in \Nat\}$. We say that $(X,d,(O_n)_{n\in \Nat})$ is almost universal and ultrahomogeneous if for any Polish metric space $(Y,d,(F_n)_{n\in \Nat})$, where $(F_n)$ are of the same type as $(O_n)$ there is an isometric embedding $\phi:Y\hookrightarrow X$ and an injection $\pi:\Nat\rightarrow \Nat$ such that $\phi$ maps $F_n$ into $O_{\pi(n)}$. Moreover, let $F_1,F_2$ be two finite subspaces of $X$ such that there is a finite isometry $\phi$ between $F_1$ and $F_2$ and two sets of indices $\{k_1,\ldots,k_n\}\subseteq \Nat$ and $\{l_1,\ldots,l_n\}\subseteq \Nat$ such that $\phi$ maps the restriction $O_{k_i}\upharpoonright F_1$ into the restriction $O_{l_i}\upharpoonright F_2$, for all $i\leq n$. Then there is an isometry $\bar{\phi}\supseteq \phi$ of the whole space $X$ extending $\phi$ and a bijection $\pi:\Nat\rightarrow \Nat$, such that $\pi(k_i)=l_i$ for $i\leq n$, such that $\bar{\phi}$ maps $O_m$ into $O_{\pi(m)}$ for all $m\in \Nat$.
\end{defin}
We remark that in all cases the underlying Polish metric space $X$ for a given structure is isometric to the Urysohn universal space $\Ur$, thus from now on we will always denote it as $\Ur$. We will comment on this in Remark \ref{isUrysohn} after the proof of Theorem \ref{main}.\\

\noindent {\bf Notational convention.} For any metric space $X$ we will denote the metric as either $d_X$ but more often, when there is no danger of confusion, just as $d$. When working with a metric on a product of metric spaces we always consider the sum metric, i.e. $d((x_1,x_2),(y_1,y_2))=d(x_1,y_1)+d(x_2,y_2)$.

We usually denote tuples $(x_1,\ldots,x_m)$, for an arbitrary $m\in \Nat$ clear from the context, by $\vec{x}$. When $\phi$ is some mapping we denote $(\phi(x_1),\ldots,\phi(x_m))$ by $\phi^m(\vec{x})$.
\section{Universal closed relations}
In this section we consider Polish metric spaces equipped with closed relations on their products. Our starting theorem is the following.
\begin{thm}\label{basic}
Let $n_1\leq \ldots \leq n_m$ be an arbitrary non-decreasing sequence of natural numbers. Then there exist closed relations (subsets) $F_{n_i}\subseteq \Ur^{n_i}$, for $i\leq m$, such that the structure $(\Ur,F_{n_1},\ldots,F_{n_m})$ is universal and ultrahomogeneous and it is unique (up to isometry preserving the relations) with this property.
\end{thm}
Instead of giving a proof of this theorem we prove the theorem below which is ``almost" more general. In remarks after the proof of Theorem \ref{main} we indicate how to modify the proof so that it works also for Theorem \ref{basic}.
\begin{thm}\label{main}
There exists an almost universal and ultrahomogeneous structure $(\Ur, (F_m^n)_{n,m\in \Nat})$ where $F_m^n\subseteq \Ur^n$ is a closed $n$-ary relation (i.e. a closed subset of the $n$-th power of $\Ur$). It is also unique with this property (up to permutation of the set of $n$-ary relations for each $n$ and isometry preserving the relations).
\end{thm}
\begin{remark}
Let us elaborate more on the statement of the theorem. Let $(X,d)$ be a Polish metric space equipped with closed sets $G_m^n$, for all $m,n\in \Nat$, where $G_m^n\subseteq X^n$. Then there exist an isometric embedding $\psi:X\hookrightarrow \Ur$ and injections $\pi_n:\Nat\rightarrow \Nat$, for each $n\in \Nat$, such that $\forall n,m\in \Nat (\psi(X)^n\cap F_{\pi_n(m)}^n=\psi^n(G_m^n))$, or in other words $\forall n,m\in \Nat\forall \vec{x}\in X^n (\vec{x}\in G_m^n\Leftrightarrow \psi^n(\vec{x})\in F_{\pi_n(m)}^n)$.

In particular, $\psi^n: G_m^n\hookrightarrow F_{\pi_n(m)}^n$ is an isometric embedding.

Moreover, let $M_1,M_2\subseteq \Ur$ be two finite metric subspaces, some $n_M\leq |M_1|=|M_2|$, for each $n\leq n_M$ there are finite sets of indices $I^{M_1}_n,I^{M_2}_n\subseteq \Nat$ such that $|I^{M_i}_n|=n_M-n+1$, for $i\in \{1,2\}$, and
for each $n\leq n_M$ there are bijections $\pi_n:I^{M_1}_n\rightarrow I^{M_2}_n$ and an isometry $\psi:M_1\rightarrow M_2$ such that $\forall n\leq n_M\forall m\in I^{M_1}_n\forall \vec{x}\in M_1^n(\vec{x}\in F_m^n\Leftrightarrow \psi^n(\vec{x})\in F_{\pi_n(m)}^n)$; i.e. $\psi$ reduces the closed relation $F_m^n$ into $F_{\pi_n(m)}^n$. Then there are an isometry $\bar{\psi}:\Ur\to \Ur$ and bijections $\bar{\pi_n}:\Nat\rightarrow \Nat$, for each $i\in \Nat$, such that $\forall n,m\in \Nat\forall \vec{x}\in \Ur^n (\vec{x}\in F_m^n\Leftrightarrow \bar{\psi}^n(\vec{x})\in F_{\bar{\pi_n}(m)}^n)$, and $\bar{\psi}$ extends $\psi$ and $\bar{\pi_n}$ extends $\pi_n$ for each $n\leq n_M$.
\end{remark}
We will construct these sets along with the underlying metric space (universal Urysohn space) as a Fra\" iss\' e limit of a certain countable class $\Age$ of finite structures. This is basically also an original method of construction of the Urysohn universal space eventhough the general Fra\" iss\' e theory did not exist at that time! We note that there is another construction of the Urysohn space due to M. Kat\v etov (\cite{Kat}).

Let us make another notational convention here. In the languages of structures that we will use there will always be defined (partial) functions into some fixed countable set, e.g. a function with rational values. It is clear that each such function can be replaced by countably many predicates; for example, a rational function $f$ can be replaced by predicates $f_q$, for each $q\in \Rat$, and then we could demand that for each element $a$ of our structure there is precisely (or at most) one $q\in \Rat$ such that $f_q(a)$ holds. We will always implicitly assume this.\\

Let $\Lng$ be a countable language consisting of $n$-ary $p_m^n$ functions with values in nonnegative rationals for every pair $m,n\in \Nat$ and binary function $d$ with values in nonnegative rationals. For any structure $A$ we will usually write just $p_m^n$ (or $d$) on $A$ instead of $(p^A)_m^n$ (or $d^A$). However, we may use the latter in few cases where there is a possibility of confusion.
\begin{defin}[The class $\Age$]\label{first_age}
A finite structure $A$ (we will not notationally distinguish a structure and its underlying set) for the language $\Lng$ of cardinality $k>0$ belongs to $\Age$ if the following conditions are satisfied:
\begin{enumerate}
\item $A$ is a rational metric space; i.e. 
\begin{itemize}
\item $d$ is a total function (defined on all pairs) on $A$
\item $\forall x,y\in A(d(x,y)=d(y,x))$
\item $\forall x,y\in A (d(x,y)=0\Leftrightarrow x=y)$
\item $\forall x,y,z\in A (d(x,y)\leq d(x,z)+d(z,y))$
\end{itemize}
Thus we will interpret $d$ as a metric.
\item There is some $n_A\leq k$ (recall that $k$ is the cardinality of $A$) such that for every $n\leq n_A$ and $m\leq n_A+1-n$ $p_m^n$ is a total function on $A$; on the other hand, for $n>n_A$ or $m>n_A+1-n$ $p_m^n$ is defined on no $n$-tuple from $A$; i.e.\\

$\exists n_A\leq k$
\begin{itemize}
\item $\forall n,m\in \Nat \forall \vec{a}\in A^n (p_m^n(\vec{a})\text{ is defined}\Leftrightarrow n\leq n_A\wedge m\leq n_A+1-n)$
\end{itemize}
We consider $p_m^n$ as a function to rationals with an interpretation that it gives a rational distance (in a ``sum" metric on $A^n$) of an $n$-tuple from one of the desired set $F_{m'}^n$. We note that $m$ and $m'$ will not necessarily be equal.
\item In order to satisfy the joint embedding property and the amalgamation property we must put some additional restrictions on these structures.
\begin{itemize}
\item $\forall m,n\in \Nat (n\leq n_A\wedge m\leq n_A+1-n\Rightarrow \forall \vec{a},\vec{b}\in A^n (p_m^n(\vec{a})\leq p_m^n(\vec{b})+d(\vec{a},\vec{b}))$
\end{itemize}
The previous formula is interpreted as follows. Consider the ''sum" metric $d$ on $A^n$, i.e. $d(\vec{a},\vec{b})=d(a_1,b_1)+\ldots +d(a_n,b_n)$. The function $p_m^n$ assigns to each $n$-tuple a non-negative rational. We interpret this function as a distance function from a fixed closed set in the sum metric. The previous formula says that a distance of some $n$-tuple from this closed set must be less or equal to the sum of a distance of another $n$-tuple from the same closed set and the distance between these two $n$-tuples. In particular, if this distance is $0$ for some $n$-tuple $\vec{a}$, i.e. we demand it will lie in the closed set, then this distance for some other $n$-tuple $\vec{b}$ must be less or equal to the distance between $\vec{a}$ and $\vec{b}$.
\end{enumerate}
\end{defin}
There is still one more condition which we must demand on these structures in order to satisfy the amalgamation property and to have only countably many isomorphism types of finite structures. We specify when we consider two structures to be isomorphic and what an embedding of one structure into another is. Informally, an isomorphism between two structures does not respect the enumeration of the rational functions $p_m^n$ for every power $n$, i.e. for example we consider structures $A=\{a_1,a_2\}$ and $B=\{b_1,b_2\}$ such that $p_1^1(a_1)=q$, $p_1^1(a_2)=0$, $p_2^1(a_1)=h$, $p_2^1(a_2)=0$ and $p_1^2$ is equal to $0$ on all pairs, and $p_1^1(b_1)=h$, $p_1^1(b_2)=0$, $p_2^1(b_1)=q$, $p_2^1(b_2)=0$ and $p_1^2$ is equal to $0$ on all pairs to be isomorphic although the roles of $p_1^1$ and $p_2^1$ are switched in these two structures.

The precise definition follows.
\begin{defin}[Isomorphism and embedding]\label{defiso}
An isomorphism between two finite structures $A,B$ in the language $\Lng$ is a pair $\phi,(\pi^\phi_n))$ where $\phi$ is an isometry between $A$ and $B$ for every $n\leq n_A$($=n_B$) $\pi^\phi_n:\{1,\ldots,n_A+1-n\}\rightarrow \{1,\ldots,n_B+1-n\}$ is a permutation such that $$\forall n\leq n_A\forall m\leq n_A+1-n \forall \vec{a}\in A^n$$ $$(p_m^n(\vec{a})=q\Leftrightarrow p_{\pi_n(m)}^n(\phi^n(\vec{a}))=q)$$ Two structures are isomorphic if there exists an isomorphism (pair) between them.

Similarly, an embedding of a structure $A$ into a structure $B$ is a pair $(\phi,(\pi_n))$ such that $\phi:A\hookrightarrow B$ is an isometric embedding  and for every $n\leq n_A$ ($\leq n_B$) $\pi_n: \{1,\ldots,n_A+1-n\}\rightarrow \{1,\ldots,n_B+1-n\}$ is an injection such that $$\forall n\leq n_A\forall m\leq n_A+1-n \forall \vec{a}\in A^n$$ $$(p_m^n(\vec{a})=q\Leftrightarrow p_{\pi_n(m)}^n(\phi^n(\vec{a}))=q)$$ 
\end{defin}

Now we must prove that $\Age$ is countable, satisfies the hereditary, joint embedding and amalgamation property.
\begin{lem}\label{isfraisse}
$\Age$ is a Fra\" iss\' e class.
\end{lem}
\begin{proof}
We will prove that $\Age$ is countable, satisfies the hereditary property, joint embedding property and amalgamation property.

For the cardinality, there are only countably many finite rational metric spaces. For each finite rational metric space $A$ of cardinality $n$ there are $n+1$ choices for $n_A$ (recall that $n_A\leq n=|A|$) and for each such a choice only finitely many rational functions $p_m^n$ can be defined, hence the claim follows.

The hereditary property is obvious. To check the joint embedding property, consider two structures $A,B\in \Age$. Let $m_A=\max\{q:(A\vDash p_q(\vec{a}))\wedge p_q\text{ is either }d_q\text{ or }p_{q,m}^n \text{ for some }n\leq n_A,m\leq n_A+1-n,\vec{a}\in A^n\}$; $m_B$ is defined similarly for $B$. Let $m=\max\{m_A,m_B\}$. Let $C=A\coprod B$ be the disjoint union of $A$ and $B$. For $a\in A$ and $b\in B$ we may set $d(a,b)=2m$, so we extend the metric on the whole $C$. To extend other predicates, we set $n_C=\max\{n_A,n_B\}$ and it is easy to see that for every $n\leq n_C$ and $m\leq n_C+1-n$ and every $n$-tuple $(c_1,\ldots,c_n)\in C^n$ on which $p_m^n$ has not been already defined there is always a choice which is consistent. For instance, for any $n\leq n_C$ and $m\leq n_C+1-n$ and $(c_1,\ldots,c_n)\in C^n$ for which $p_m^n$ has not been yet defined we may set $p_m^n(c_1,\ldots,c_n)=0$; this is consistent.

Finally, we need to check the amalgamation property. Let $A,B,C\in \Age$ be structures, we can assume WLOG that $A$ is a substructure of both $B$ and $C$ and for all $n\leq n_A$ and $m\leq n_A-n$ $(p^B)_m^n=(p^C)_m^n$. Let $D=A\coprod (B\setminus A)\coprod (C\setminus A)$. The metric is extended in a standard way, i.e. for $b\in B$ and $c\in C$ we set $d(b,c)=\min\{d(b,a)+d(a,c):a\in A\}$.

Let us set $n_D=n_B+(n_C-n_A)$ (note that $n_A\leq \min\{n_B,n_C\}$). We reenumerate some rational functions on $D$ (see Definition \ref{defiso}).
\begin{itemize}
\item For all $n\leq n_B$, $m\leq n_B+1-n$ and $\vec{b}\in B^n\subseteq D^n$ we let $(p^D)_m^n(\vec{b})=(p^B)_m^n(\vec{b})$, i.e. we keep the enumeration from the original one in $B$.
\item For all $n\leq n_A$, $m\leq n_A+1-n$ and $\vec{c}\in C^n\subseteq D^n$ we again let $(p^D)_m^n(\vec{c})=(p^C)_m^n(\vec{c})$, i.e. keep the previous enumeration.
\item For $n\leq n_A$ and $n_A+1-n<m\leq n_C+1-n$ or for $n_A<n\leq n_C$ and any $m\leq n_C+1-n$ and $\vec{c}\in C^n$ we set $(p^D)_{m+(n_B-n_A)}^n(\vec{c})=(p^C)_m^n(\vec{c})$, i.e. we change the enumeration by adding $n_B-n_A$.\\
\end{itemize}

We need to check that this metric extension along with the reenumeration of some predicates is consistent.

Specifically, we need to check that the following formula still holds true whenever the function $p_m^n$ is defined on both $\vec{b}$ and $\vec{c}$: $$\forall \vec{b},\vec{c}\in D^n (p_m^n(b_1,\ldots,b_n)\leq p_m^n(\vec{c}) +d(b_1,c_1)+\ldots +d(b_n,c_n))$$
Let some $n,m$, $\vec{b}$ and $\vec{c}$ in $D^n$ such that both $p_m^n(\vec{b})$ and $p_m^n(\vec{c})$ are defined be given.
\begin{itemize}
\item  If $n_A<n\leq n_B$ and $m\leq n_B+1-n$ then it follows that both $\vec{b}$ and $\vec{c}$ are from $B^n$ and the formula holds in $D$ since it holds in $B$.
\item For any $n$ if $m>n_B+1-n$ then it follows that both $\vec{b}$ and $\vec{c}$ are from $C^n$ and the formula holds in $D$ since it holds in $B$ with $m'=m-(n_B-n_A)$.
\item Finally, assume that $n\leq n_A$ and $m\leq n_A+1-n$. If $\vec{b}$ and $\vec{c}$ are either both from $B^n$ or both from $C^n$ then the formula holds in $D$ since it holds in $B$, resp. in $C$. So let us assume that $\vec{b}$ is originally from $B^n$ and $\vec{c}$ is originally from $C^n$ (the opposite case is the same of course). From definition, for every $i\leq n$ there is some $a_i\in A$ such that $d(b_i,c_i)=d(b_i,a_i)+d(a_i,c_i)$. We have $$p_m^n(\vec{b})\leq p_m^n(\vec{a})+d(b_1,a_1)+\ldots +d(b_n,a_n)$$ since this formula holds true in $B$. Similarly, we have $$p_m^n(\vec{a})\leq p_m^n\vec{c})+d(a_1,c_1)+\ldots +d(a_n,c_n)$$ since this formula holds true in $C$. Putting together, we obtain $$p_m^n(\vec{b})\leq p_m^n(\vec{c}) +d(b_1,c_1)+\ldots +d(b_n,c_n)$$ which is what we wanted to prove.\\
\end{itemize}
For any other $m,n$ that were not listed above the functions $p_m^n$ were not yet defined. So for $n\leq n_D$ and $m\leq n_D+1-n$ that were not listed above we may set $p_m^n(\vec{d})=0$ for all $\vec{d}\in D^n$ for instance. For any fixed pair $n\leq n_D$ and $m\leq n_D+1-n$ that was listed above but $p_m^n$ was not yet defined on some $\vec{d}\in D^n$ we define it canonically as follows (but there are other possible definitions too): We set $p_m^n(\vec{d})=\max \{0,\max \{p_m^n(\vec{d'})-d(\vec{d'},\vec{d}):p_m^n\text{ was defined on }\vec{d'}\}\}$. Note that even if we used this definition on some $n$-tuple on which $p_m^n$ had been already defined then it would get the same value. That is why we call it canonical. To check it is consistent let $\vec{d_0},\vec{d_1}\in D^n$ be some $n$-tuples. Assume at first that $p_m^n(\vec{d_0})>0$ and $p_m^n(\vec{d_1})>0$ and let $\vec{d'_0},\vec{d'_1}\in D^n$ be such that $p_m^n(\vec{d_0})=p_m^n(\vec{d'_0})-d(\vec{d'_0},\vec{d_0})$ and $p_m^n(\vec{d_1})=p_m^n(\vec{d'_1})-d(\vec{d'_1},\vec{d_1})$. Then $p_m^n(d_0)=p_m^n(\vec{d'_0})-d(\vec{d'_0},\vec{d_0})\leq p_m^n(\vec{d'_0})-d(\vec{d'_0},\vec{d_1})+d(\vec{d_0},\vec{d_1})\leq p_m^n(\vec{d'_1})+d(\vec{d_0},\vec{d_1})$. The case when $p_m^n(\vec{d_0})=0$ or $p_m^n(\vec{d_1})=0$ is similar and the proof is left to the reader.
\end{proof}
Since $\Age$ is a Fra\" iss\' e class it has a Fra\" iss\' e limit which we denote $U$. Besides other things it is a metric space. In fact it is a countable universal homogeneous rational metric space (see Remark \ref{isUrysohn}). By $\Ur$ we denote its metric completion which is the universal Urysohn space. For every natural $n$ we also have the set $\mathcal{F}_n$ of countably many rational functions on $U^n$ without an enumeration arising from the Fra\" iss\' e limit though. We choose some enumeration and denote the set $\mathcal{F}_n$ as $\{f_m^n:m\in \Nat\}$ for every $n$.
For every $m,n\in \Nat$ the set $\tilde{F}_m^n$ of all $n$-tuples $\vec{u}$ from $U$ such that $f_m^n(\vec{u})=0$ is a closed subset of $U^n$. By $F_m^n$ we denote the closure of $\tilde{F}_m^n$ in the completion $\Ur$ (thus we have $F_m^n\cap U=\tilde{F}_m^n$). This finishes the construction of the sets from the statement of Theorem \ref{main}. We must now prove the almost universality and ultrahomogeneity of these sets which we do in the following section.
\subsection{One-point extension property}
When constructing the Fra\" iss\' e limit we had to work only with rational metric spaces and rational functions $p_m^n$ on them in order to have the class $\Age$ countable and to have the limit $U$. The one-point extension property holds for substructures of $U$ (see \ref{OPEFr}). We formulate it here for convenience again. We call it here ``rational one-point extension property".\\

\noindent{\bf Rational one-point extension property.} Let $A\in \Age$ be a finite rational metric space such that the rational functions $f_m^n$, for $n\leq n_A\leq |A|$ and $m\leq n_A+1-n$, are defined on it. Let $B\in \Age$ be a one point extension of $A$, i.e. $|B|-|A|=1$ and there is an embedding $(\iota,(\pi_n^\iota)): A\hookrightarrow B$. Assume that there is an embedding $(\phi,(\pi_n^\phi)): A\hookrightarrow U$. Then there is an embedinng $(\psi,(\pi_n^\psi)): B\hookrightarrow U$ extending $(\phi,(\pi_n^\phi))$, i.e. $(\phi,(\pi_n^\phi))=(\psi,(\pi_n^\psi)) \circ (\iota,(\pi_n^\iota))$.\\

Before we proceed further we use this place for the following remark.
\begin{remark}\label{isUrysohn}
We still owe the explanation that the underlying metric space of our (almost) universal and ultrahomogeneous structure is isometric to the Urysohn universal metric space. To prove it it suffices to check that the underlying metric structure $U$ of the countable Fra\" iss\' e limit is isometric to the universal rational metric space (as its completion is isometric to the Urysohn space). However, realize that a countable rational metric space $X$ is isometric to the universal rational metric space if and only if it has the rational one-point metric extension property: for any finite metric subspace $F\subseteq X$ and any one-point extension $G\supseteq F$ which is still a rational metric space, there is an isometric embedding $\iota :G\hookrightarrow X$ such that $\iota\upharpoonright F=\mathrm{id}$.

However, $U$ has this rational one-point metric extension property. Here, and also in the next section, its rational one-point extension property is always stronger.
\end{remark}
However, since we made the completion $\Ur$ we want to have this kind of one-point extension property for all finite substructures of $\Ur$, not just for those that are actually substructures of $U$.
In this section we prove this full one-point extension property. The almost universality and homogeneity, and uniqueness will follow by a standard argument. We define a generalized class $\bar{\Age}$ of structures (that correspond to finite substructures of $(\Ur,(F_m^n))$).
\begin{defin}
A substructure $A\in \bar{\Age}$ is a finite metric space, moreover there is some $n_A\leq |A|$ and for each $n\leq n_A$ there is a finite set of indices $I^A_n\subseteq \Nat$ such that $|I^A_n|=n_A-n+1$. For each $n\leq n_A$ and $m\in I^A_n$ there is a closed subset $G_m^n\subseteq A^n$. By $p_m^n$ we shall again denote the distance function from the set $G_m^n$.

An embedding of a substructure $A$ into a substructure $B$ is a pair $(\phi,(\pi_n))$ where $\phi:A\hookrightarrow B$ is an isometric embedding and for each $n\leq n_A$ $\pi_n:I^A_n\rightarrow I^B_n$ is an injection such that $\forall n\leq n_A\forall m\in I^A_n\forall \vec{x}\in A^n (p_m^n(\vec{x})=p_{\pi_n(m)}^n(\phi^n(\vec{x})))$.
\end{defin}
Thus we just drop the condition that the metric and functions $p_m^n$ have to have rational values.
\begin{prop}[One-point extension property]\label{main_OPE}
Let $A$ be a finite substructure of $(\Ur,(F_m^n))$ and let $B\in\bar{\Age}$ be such that $|B|=|A|+1$ and there is an embedding $(\phi,(\pi_n^\phi))$ of $A$ into $B$. Then there exists an embedding $(\psi,(\pi_n^\psi))$ of $B$ into $(\Ur,(F_m^n))$ such that $\mathrm{id}=(\psi,(\pi_n^\psi))\circ (\phi,(\pi_n^\phi))$.
\end{prop}
Before we provide a proof we show that the almost universality and homogeneity and also the uniqueness follow from Proposition \ref{main_OPE}.

\begin{claim}[Almost universality]
$(\Ur,(F_m^n))$ is almost universal.
\end{claim}
\begin{proof}[Proof of the Claim]
Let $(X,d)$ be a Polish metric space (in fact, it can be just separable metric) equipped with sets $(G_m^n)_{m,n\in \Nat}$ where for each $n$ and $m$ $G_m^n\subseteq X^n$ is a closed subset of the $n$-th power of $X$. Let $D\subseteq X$ be a countable subset with the following properties:
\begin{itemize}
\item $D$ is a dense subset of $X$
\item For every $m$ and $n$ $D^n\cap G_m^n$ is a dense subset of $G_m^n$.
\end{itemize}
We prove that there exist an isometric copy $D'$ of $D$ in $\Ur$ and injections $\pi_i$ from $\Nat$ to $\Nat$ for all $i$ such that for every $m$ and $n$ and $\vec{d}\in D^n$ we have $\vec{d}\in G_m^n\Leftrightarrow \vec{d'}\in F_{\pi_n(m)}^n$ and if $\vec{d}\notin G_m^n$ then $d(\vec{d},G_m^n\}=d_\Ur(\vec{d'},F_{\pi_n(m)}^n)$, where $\vec{d'}$ corresponds to $\vec{d}$ in the copy. Then we will extend the isometry to the closure of $D$ which is the whole space $X$ and we will be done. To see that, let $m,n\in \Nat$ and $\vec{x}\in X^n$ be arbitrary.

If $\vec{x}\in G_m^n$ then there is a sequence $(d_1^j,\ldots,d_n^j)_j\subseteq D^n$ converging to $\vec{x}$ such that $\vec{d^j}\in G_m^n$ for every $j$. From our assumption, $\vec{{d'}^j}\in F_{\pi_n(m)}^n$ and since $F_{\pi_n(m)}^n$ is closed the image of $\vec{x}$ also lies in $F_{\pi_n(m)}^n$.

If $\vec{x}\notin G_m^n$ and $\varepsilon=d(\vec{x},G_m^n)$ then there is $\vec{d}\in D^n$ such that $d(\vec{x},\vec{d})<\varepsilon/3$. It follows that $d(\vec{d},G_m^n)>2\varepsilon/3$, thus $d_\Ur(\vec{d'},F_{\pi_n(m)}^n)>2\varepsilon/3$ and thus the image of $\vec{x}$ also does not lie in $F_{\pi_n(m)}^n$.\\

Let us enumerate the set $D$ as $\{d_1,d_2,\ldots\}$. The construction of $D'$ is by induction, just a series of applications of Proposition \ref{main_OPE}. Let $B_1$ be a one-point structure containing $d_1$, $n_{B_1}=1$ and $I^{B_1}_1=\{1\}$. Let $A_1$ be an empty structure and use Proposition \ref{main_OPE} to get an embedding of $B_1$ into $\Ur$. The embedding determines a point $u_1\in \Ur$ and also an injection $\pi_1:I^{B_1}_1\rightarrow \Nat$. We have $d(d_1,G_1^1)=p_{\pi_1(1)}^1(u_1)=d_\Ur(u_1,F_{\pi_1(1)}^1)$.

Assume we have found $u_1,\ldots,u_{k-1}$. Consider a structure $B_k$ containing $\{d_1,\ldots,d_k\}$, $n_{B_k}=k$, for $i\leq k$ $I^{B_k}_i=\{1,\ldots,k-i+1\}$. Let $A_k$ be a substructure of $(\Ur,(F_m^n))$ containing $\{u_1,\ldots,u_{k-1}\}$, $n_{A_k}=k-1$ and for $i\leq k-1$ $I^{A_k}_i=\{\pi_i(1),\ldots,\pi_i(k-i)\}$. There is an obvious embedding of $A_k$ into $B_k$ so we can use Proposition \ref{main_OPE} to extend $A_k$ by some new point $u_k$. We also extend the domain of $\pi_i$, for $i\leq k-1$, by $k-i+1$ and obtain a new injection $\pi_k$ with domain $\{1\}$. This finishes the induction.
\end{proof}
\begin{claim}[Almost ultrahomogeneity]
$(\Ur,(F_m^n))$ is almost ultrahomogeneous.
\end{claim} 
\begin{proof}[Sketch of the proof]
Let $A$ and $B$ be two isomorphic substructures (witnessed by $(\phi,(\pi_n^\phi))$) of $(\Ur,(F_m^n))$. WLOG assume that for every $n\leq n_A=n_B$ we have $I_n^A=I_n^B=\{1,\ldots,n_A-n+1\}$ and $\pi_n^\phi$ is the identity on $I_n^A$. Let $D=\{u_n:n\in \Nat\}\subseteq U$ be a countable dense subset such that for every $m,n$ $D^n\cap F_m^n$ is dense in $F_m^n$. By a back-and-forth series of use of the one-point extension property (Proposition \ref{main_OPE}) we shall be extending the isomorphism $(\phi,(\pi_n^\phi))$ into a chain $(\phi,(\pi_n^\phi))\subseteq (\phi_1,(\pi_{n,1}^{\phi_1}))\subseteq (\phi_2,(\pi_{n,2}^{\phi_2}))\subseteq \ldots$ so that for every $m\in \Nat$ $u_m$ is both in the domain and range of $\phi_m$ and $m$ is in the domain and range of $\pi_{m,1}$. $\bigcup _m (\phi_m,(\pi_{n,m}^{\phi_m}))$ is the desired isomorphism of $(\Ur,(F_m^n))$.
\end{proof}
\begin{claim}[Uniqueness]
$(\Ur,(F_m^n))$ is unique with the almost universality and ultrahomogeneity property.
\end{claim}
This is again done by a standard back-and-forth argument using Proposition \ref{main_OPE}.

Before we prove Proposition \ref{main_OPE} we need the following lemma that will be useful in the next section too.
\begin{lem}\label{metriclem}
Let $M=\{d_1,\ldots,d_k\}$ be a given finite metric space. Also, for every $i< k$ let $(u_i^j)_j\subseteq U$ be a given rational Cauchy sequence from the rational Urysohn space such that $d(u_i^j,u_i^{j+1})\leq 1/2^{j+1}$ for all $j$ and moreover, $d_\Ur(\lim_n u_i^n,\lim_n u_j^n)=d_M(d_i,d_j)$ for every $i,j< k$.

Moreover, let $l\in \Nat$ be given and let $\{u_k^1,\ldots,u_k^{l-1}\}\subseteq U$ (if $l=1$ then it is an empty sequence) be a given finite rational sequence with the following property: for every $j<l$ and every $i< k$ we have $d_M(d_k,d_i)+1/(k\cdot 2^{j+1})\leq d(u_k^j,u_i^{j+k+2})\leq d_M(d_k,d_i)+1/2^j$. 

Then if we consider the space $A_k=\{u_1^{l+k+2},\ldots,u_{k-1}^{l+k+2},u_k^{l-1}\}$ (resp. $A_k=\{u_1^{l+k+2},\ldots,u_{k-1}^{l+k+2}\}$ if $l=1$) then there exists a rational metric extension $U\supseteq M_k=A_k\cup\{g_k\}$ such that $d_M(d_k,d_i)+(2i-1)/(k\cdot 2^{l+1})\leq d(g_k,u_i^{l+k+2})\leq d_M(d_k,d_i)+(2i)/(k\cdot 2^{l+1}$ for all $i<k$ and if $l>1$ then also $d(g_k,u_k^{l-1})=1/2^l$.
\end{lem}
\begin{proof}[Proof of the lemma.]\emph{}\\
We will treat separately two cases. Case 1 is when $l=1$ and Case 2 is when we are moreover given a non-empty finite sequence $\{u_k^1,\ldots,u_k^{l-1}\}$, i.e. $l>1$.\\

\noindent{\bf Case 1:} $l=1$.

Let $i_1,\ldots,i_{k-1}$ be a permutation of $\{1,\ldots,k-1\}$ such that we have $d(d_k,d_{i_1})\geq d(d_k,d_{i_2})\geq \ldots\geq d(d_k,d_{i_{k-1}})$. For each $j< k$ we shall denote $v_j$ the element $u_j^{l+k+2}$. We have that $d_\Ur(v_j,u_j)\leq 1/2^{l+k+2}$. We now work with $\{v_1,\ldots,v_{k-1}\}$. For $j< k$ let $\gamma_j\in \Rea^+$ be arbitrary positive real numbers such that $(2j-1)/(k\cdot 2^{l+1})\leq \gamma_j \leq (2j)/(k\cdot 2^{l+1})$ and $\eta_j=d(d_k,d_{i_j})+\gamma_j\in \Rat$. We claim there exists $g_k\in U$ such that $d_\Ur(g_k,v_j)=\eta_j$. We just need to check that the triangle inequalities are satisfied, then it will follow that such an element $g_k$ does exist from the one-point (metric) extension property of $U$.

Let $i<j< k$, we shall check that $\eta_i-\eta_j\leq d(v_{i_i},v_{i_j})\leq \eta_i+\eta_j$. We have $|d(v_{i_i},v_{i_j})-d(d_{i_i},d_{i_j})|<1/2^{l+k+1}\leq 1/(k\cdot 2^l)$. Since $\eta_i-\eta_j\leq d(d_k,d_{i_i})-d(d_k,d_{i_j})-1/(k\cdot 2^{l+1})\leq d(d_k,d_{i_i})-d(d_k,d_{i_j})-1/2^{l+k+1}$, thus $\eta_i-\eta_j\leq d(v_{i_i},v_{i_j})$. Since $\eta_i+\eta_j\geq d(d_k,d_{i_i})+d(d_k,d_{i_j})+1/(k\cdot 2^{l+1})\geq d(d_k,d_{i_i})+d(d_k,d_{i_j})+1/2^{l+k+1}$, thus also $d(v_{i_i},v_{i_j})\leq \eta_i+\eta_j$.

So by the one-point extension there exists such $g_k\in U$.\\

\noindent{\bf Case 2:} $l>1$. We proceed identically as in Case 1, we just need to care about the element $u_k^{l-1}$. Let again $i_1,\ldots,i_{k-1}$ be a permutation of $\{1,\ldots,k-1\}$ such that we have $d(d_k,d_{i_1})\geq d(d_k,d_{i_2})\geq \ldots\geq d(d_k,d_{i_{k-1}})$. For each $j< k$ we shall denote $v_j$ the element $u_j^{l+k+2}$. We work with the space $\{u_k^{l-1},v_1,\ldots,v_{k-1}\}$. For $j< k$ let $\gamma_j\in \Rea^+$ be arbitrary positive real numbers such that $(2j-1)/(k\cdot 2^{l+1})\leq \gamma_j \leq (2j)/(k\cdot 2^{l+1})$ and $\eta_j=d(d_k,d_{i_j})+\gamma_j\in \Rat$. We claim there exists $g_k\in U$ such that $d_\Ur(g_k,v_j)=\eta_j$ and moreover $d_\Ur(g_k,u_k^{l-1})=1/2^l$. We again just need to check that the triangle inequalities are satisfied, then it will follow that such an element $g_k$ does exist.

For $i<j< k$ the verification that $\eta_i-\eta_j\leq d(v_{i_i},v_{i_j})\leq \eta_i+\eta_j$ holds is the same as in Case 1.

Now let $j< k$ be given. We need to check that $\eta_j-1/2^l\leq d(v_{i_j},u_{k+1}^{l-1})\leq \eta_j+1/2^l$. Note that $$d(u_k^{l-1},u_{i_j}^{k+l+1})-d(u_{i_j}^{k+l+1},v_{i_j})\leq d(v_{i_j},u_k^{l-1})$$ and $$d(v_{i_j},u_k^{l-1})\leq d(u_k^{l-1},u_{i_j}^{k+l+1})+d(u_{i_j}^{k+l+1},v_{i_j})$$
The following estimates on $d(u_k^{l-1},u_{i_j}^{k+l+1})$ follow from the assumption from the statement of the lemma. We have $$d(d_{i_j},d_k)+(2j-1)/(k\cdot 2^l)\leq d(u_k^{l-1},u_{i_j}^{k+l+1})\leq d(d_{i_j},d_k)+(2j)/(k\cdot 2^l)$$
Similarly, we have the following estimates on $\eta_j$:
$$d(d_{i_j},d_k)+(2j-1)/(k\cdot 2^{l+1})\leq \eta_j\leq d(d_{i_j},d_k)+(2j)/(k\cdot 2^{l+1})$$
We check the inequality $\eta_j-1/2^l\leq d(v_{i_j},u_k^{l-1})$. Putting the previous inequalities together it suffices to check that $$d(d_{i_j},d_k)+(2j)/(k\cdot 2^{l+1})-1/2^l\leq d(d_{i_j},d_k)+(2j-1)/(k\cdot 2^l)-1/2^{k+l+2}$$ By subtracting from both sides we get $$(-2j+2)/(k\cdot 2^{l+1})-1/2^l\leq -1/2^{k+l+2}$$ which clearly holds.

To check the other inequality $d(v_{i_j},u_k^{l-1})\leq \eta_j+1/2^l$ using the previous inequalities it suffices to check that $$d(d_{i_j},d_k)+(2j)/(k\cdot 2^l)+1/2^{k+l+2}\leq d(d_{i_j},d_k)+(2j-1)/(k\cdot 2^{l+1})+1/2^l$$ By subtracting from both sides we get $$(2j+1)/(k\cdot 2^{l+1})+1/2^{k+l+2}\leq 1/2^l$$ Since $j\leq k-1$ we have $$(2j+1)/(k\cdot 2^{l+1})+1/2^{k+l+2}\leq (2k-1)/(k\cdot 2^{l+1})+1/2^{k+l+2}$$
and the following equality holds
$$(2k-1)/(k\cdot 2^{l+1})+1/2^{k+l+2}=1/2^l-1/(k\cdot 2^{l+1})+1/2^{k+l+2}$$ The right hand side is clearly less or equal to $1/2^l$ so we are done.

So again by the one-point (metric) extension property there exists such $g_k\in U$.
\end{proof}
\begin{proof}[Proof of Proposition \ref{main_OPE}]
Let us at first treat the case when $A$ is empty and $B$ is a one-point structure $\{b_1\}$. We have $n_B=1$ and WLOG assume that $I_1^B=\{1\}$. Thus we only need to find some $a_1\in \Ur$ and $m\in \Nat$ such that $p_m^1(a_1)=p_1^1(b_1)$. For every $n\in \Nat$ let $\delta_n\in \Rat_0^+$ be any non-negative rational number such that $p_1^1(b_1)\leq \delta_n\leq p_1^1(b_1)+1/2^{l+2}$. We use the rational one-point extension property to define a sequence $(u_1^j)_j\subseteq U$ and to obtain $m\in \Nat$ such that for every $j\in \Nat$ $p_m^1(u_1^j)=\delta_j$ and $d_\Ur(u_1^j,u_1^{j+1})=1/2^{j+1}$. It is straightforward to check that we have $p_m^1(a_1)=p_1^1(b_1)$ where $a_1$ is the limit of the sequence $(u_1^j)_j$.\\

We now assume that $A$ is non-empty. Let us enumerate $A$ as $\{a_1,\ldots,a_{k-1}\}$ and $B$ as $\{b_1,\ldots,b_k\}$ so that the embedding ($(\phi,(\pi^\phi_n))$ of $A$ into $B$ sends $a_i$ to $b_i$ for every $i< k$. We extend $A$ by adding a point $a_k$. We will find a Cauchy sequence of elements from $U$ such that the limit will be this desired point $a_k$. For each $l< k$ let us choose a converging sequence $(u_l^j)_j\subseteq U$ of elements from the Fra\" iss\' e limit such that $\lim_j u_l^j=a_l$, $d_\Ur(u_l^j,a_l)<1/2^j$ and for $i<j$ we have $d_\Ur(u_l^j,a)<d_\Ur(u_l^i,a)$.

In order to simplify the notation we assume that $n_B=n_A+1$ and for each $n\leq n_A$ $I^A_n=\{1,\ldots,n_A-n+1\}$ and also for each $n\leq n_B$ $I^B_n=\{1,\ldots,n_B-n+1\}$ and the injections $\pi^\phi_n$ are the identities.  Consider a structure $S_1=\{u_1^{k+3},\ldots,u_{k-1}^{k+3}\}$ with $n_{S_1}=n_A$ and for every $n\leq n_{S_1}$, $m\leq n_{S_1}-n+1$ and $\vec{x}\in S_1^n$ $p_m^n(\vec{x})=d_\Ur(\vec{x},F_m^n)$. Thus $S_1\in \Age$ and for any $i,j< k$ we have $|d_\Ur(u_i^{k+3},u_j^{k+3})-d(b_i,b_j)|<1/2^{k+2}$. We use Lemma \ref{metriclem} to define a metric one-point extension $M_1=\{u_1^{k+3},\ldots,u_{k-1}^{k+3},g\}$ of $S_1$ such that for all $i< k$ we have $d(b_i,b_k)\leq d_\Ur(u_i^{k+3},g)\leq d(b_i,b_k)+1/2$. We define a structure $V_1$ with $n_{V_1}=n_{S_1}+1=n_B$ such that $M_1$ is its underlying (rational) metric space. We need to define (rational) $p_m^n$ on all $n$-tuples containing $g$ for all $n\leq n_B$ and $m\leq n_B-n+1$ and also on all $n$-tuples (not necessarily containing $g$) for $n\leq n_B$ and $m=n_B-n+1$ to obtain a one-point extension $V_1$ of $S_1$.

Fix such a pair $n,m$. Let us enumerate all $n$-tuples $\vec{x}\in M_1^n$ as $(\vec{x}_j^1)_{j<J}$ so that all $n$-tuples not containing $g$ precede every $n$-tuple containing $g$.
Also, for any $n$-tuple $\vec{x}\in M_1^n$ let $\vec{b}_{\vec{x}}$ denote the corresponding $n$-tuple $\vec{y}$ from $B^n$ (via the function sending $u_i^{k+3}$ to $b_i$ for $i< k$ and $g$ to $b_k$). We inductively define $p_m^n$ on $\vec{x}_j^1$'s. Let $\vec{x}_j^1$, for some $j<J$, be given. Let $\varepsilon_j^1=p_m^n(\vec{b}_{\vec{x}_j^1})$. It is not necessarily a rational number. Let $r_j^1\in \Rat$ be an arbitrary rational number such that $\varepsilon_j^1\leq r_j^1\leq \varepsilon_j^1+n/2^{k+3}$. Also, let $m_j^1=\max\{p_m^n(\vec{x})-d(\vec{x}_j^1,\vec{x}):\vec{x}\in M_1^n\wedge p_m^n\text{ has been already defined on }\vec{x}\}$ and $M_j^1=\min\{p_m^n(\vec{x})+d(\vec{x}_j^1,\vec{x}):\vec{x}\in M_1^n\wedge p_m^n\text{ has been already defined on }\vec{x}\}$. If $m_j^1\leq r_j^1\leq M_j^1$ then we set $p_m^n(\vec{x}_j^1)=r_j^1$. If $r_j^1<m_j^1$, resp. $r_j^1>M_j^1$ then we set $p_m^n(\vec{x}_j^1)=m_j^1$, resp. $p_m^n(\vec{x}_j^1)=M_j^1$. Note that if $n\leq n_A$ and $m\leq n_A-n+1$ and $\vec{x}_j^1\in S_1^n$ then $m_j^1=M_j^1=p_m^n(\vec{x}_j^1)$, thus by our assigning we really obtain an extension of $S_1$. Thus by a weak one-point extension property we obtain some $u_k^1\in U$ playing the role of $g$. 

Assume we have already constructed $u_k^1,\ldots,u_k^{l-1}\subseteq U$ such that $d_\Ur(u_k^i,u_k^{i+1})=1/2^{i+1}$ for $0\leq i<l-1$. Consider a structure $S_l=\{u_1^{k+l+2},u_{k-1}^{k+l+2},u_k^{l-1}\}$ with $n_{S_l}=n_B$ and for every $n\leq n_{S_l}$, $m\leq n_{S_l}-n+1$ and $\vec{x}\in S_l^n$ $p_m^n(\vec{x})=d_\Ur(\vec{x},F_m^n)$.
Thus $S_l\in \Age$ and for any $i,j< k$ we have $|d_\Ur(u_i^{k+l+2},u_j^{k+l+2})-d(b_i,b_j)|<1/2^{k+l+1}$. We again use Lemma \ref{metriclem} to obtain a metric one-point extension $M_l=\{u_1^{k+l+2},u_{k-1}^{k+l+2},u_k^{l-1},g\}$ of $S_l$ such that such that for all $i< k$ we have $d(b_i,b_k)\leq d_\Ur(u_i^{k+l+2},g)\leq d(b_i,b_k)+1/2^l$. 

For $n\leq n_B$ and $m\leq n_B-n+1$ we need to define $p_m^n$ on all $n$-tuples from $M_l^n$ containing the new element $g$. We do it as before: Fix such a pair $n,m$. Let us again enumerate all $n$-tuples $\vec{x}\in M_l^n$ as $(\vec{x}_j^l)_{j<K}$ so that all $n$-tuples not containing $g$ precede any $n$-tuple containing $g$.
Also, for any $n$-tuple $\vec{x}\in M_l^n$ let again $\vec{b}_{\vec{x}}$ denote the corresponding $n$-tuple $\vec{y}$ from $B^n$ (via the function sending $u_i^{k+l+2}$ to $b_i$ for $i< k$ and $u_k^{l-1}$ and $g$ to $b_k$). We inductively define $p_m^n$ on $\vec{x}_j^l$'s. Let $\vec{x}_j^l$, for some $j<K$, be given. Let $\varepsilon_j^l=p_m^n(\vec{b}_{\vec{x}_j^l})$. It is not necessarily a rational number. Let $r_j^l\in \Rat$ be an arbitrary rational number such that $\varepsilon_j^l\leq r_j^l\leq \varepsilon_j^l+n/2^{k+l+2}$. Also, let $m_j^l=\max\{p_m^n(\vec{x})-d(\vec{x}_j^l,\vec{x}):\vec{x}\in M_l^n\wedge p_m^n\text{ has been already defined on }\vec{x}\}$ and $M_j^l=\min\{p_m^n(\vec{x})+d(\vec{x}_j^l,\vec{x}):\vec{x}\in M_l^n\wedge p_m^n\text{ has been already defined on }\vec{x}\}$. If $m_j^l\leq r_j^l\leq M_j^l$ then we set $p_m^n(\vec{x}_j^l)=r_j^l$. If $r_j^l<m_j^l$, resp. $r_j^l>M_j^l$ then we set $p_m^n(\vec{x}_j^l)=m_j^l$, resp. $p_m^n(\vec{x}_j^l)=M_j^l$.This is again a consistent extension of $S_l$. Thus by a weak one-point extension property we obtain some $u_k^l\in U$ playing the role of $g$.

Assume the induction is finished. We have found a sequence $(u_k^j)_j$. Moreover, realize that for every $n\leq n_B$ and $m=n_B-n+1$ there is some $\varpi_n\in \Nat$ such that for every $\vec{x}\in \{u_i^j:i\leq k,j\in \Nat\}^n$ we have $p_m^n(\vec{x})=d_\Ur(\vec{x},F_{\varpi_n}^n)$. Since for any $j\in \Nat$ we have $d_\Ur(u_k^j,u_k^{j+1})=1/2^{j+1}$, this sequence is Cauchy with a limit that we denote $a_k$. We define an embedding $(\psi,(\pi_n^\psi))$ of $B$ into $(\Ur,F_m^n)$ as follows: $\psi(b_i)=a_i$ for every $i\leq k$ and for $n\leq n_B$ we set $\pi_n^\psi(i)=i$ if $i<n_B-n+1$ and $\pi_n^\psi(i)=\varpi_n$ if $i=n_B-n+1$. It follows from the use of Lemma \ref{metriclem} that $d_\Ur(a_i,a_k)=d(b_i,b_k)$ for every $i<k$. We must check that $p_n^m(\vec{x})=p_{\pi_n^\psi(m)}^n(\psi^n(\vec{x}))$ for all $n\leq n_B$, $m\leq n_B-n+1$ and $\vec{x}\in B^n$.
\begin{claim}\label{correct_estim_main}
For every $j<K$ and $l\in \Nat$ we have $\varepsilon_j^l-n/2^{k+l+2}\leq p_m^n(\vec{x}_j^l)\leq \varepsilon_j^l+n/2^{k+l+2}$.
\end{claim}
Once the claim is proved the assertion follows. So it remains to prove the claim.\\

\noindent \emph{Proof of the Claim.} We prove it for every $j<K$ by induction on $l$.\\

\noindent{\bf Step 1.}\\
Suppose $l=1$. Let us prove that $p_m^n(\vec{x}_j^1)\leq \varepsilon_j^1+n/2^{k+3}$. We have $p_m^n(\vec{x}_j^1)=\max\{r_j^1,m_j^1\}$. Since $r_j^1\leq \varepsilon_j^1+1/2^{1+1}$ it suffices to prove that $m_j^1\leq \varepsilon_j^1+(2n+1)/2^{1+1}$.

Realize that $m_j^1=p_m^n(\vec{x}_p^1)-d(\vec{x}_j^1,\vec{x}_p^1)$ for some $\vec{x}_p^1$. 
\begin{enumerate}
\item There exists $\vec{x}_p^1\in S_1^n$ such that $m_j^1=p_m^n(\vec{x}_p^1)-d(\vec{x}_j^1,\vec{x}_p^1)$. Let $\vec{x}_p^1=(u_{i_1}^{k+3},\ldots,u_{i_n}^{k+3})$. Since for every $m\leq n$ we have $d(u_{i_m}^{k+3},a_{i_m})\leq 1/2^{k+3}$, we have that $d(\vec{x}_p^1,(a_{i_1},\ldots,a_{i_n}))\leq n/2^{k+3}$, thus $p_m^n(\vec{x}_p^1)\leq \varepsilon_p^1+n/2^{k+3}$. We also have that $d(\vec{b}_{\vec{x}_p^1},\vec{b}_{\vec{x}_j^1})\leq d(\vec{x}_p^1,\vec{x}_j^1)$. Finally, since $\varepsilon_p^1\leq \varepsilon_j^1+d(\vec{b}_{\vec{x}_p^1},\vec{b}_{\vec{x}_j^1})$, putting the inequalities together we obtain $m_j^1\leq \varepsilon_j^1+n/2^{k+3}$.
\item There does not exist such $\vec{x}_p^1\in S_1^n$. We claim that then $m_j^1=p_m^n(\vec{x}_p^1)-d(\vec{x}_j^1,\vec{x}_p^1)$ where $p_m^n(\vec{x}_p^1)=r_p^1$. Once we prove this is true then from the same series of inequalities as in the item above we prove the desired inequality. Suppose it is not true. Then $m_j^1=p_m^n(\vec{x}_p^1)-d(\vec{x}_j^1,\vec{x}_p^1)$ and $p_m^n(\vec{x}_p^1)=p_m^n(\vec{x}_{q_1}^1)-d(\vec{x}_p^1,\vec{x}_{q_1}^1)$ for some $\vec{x}_{q_1}^1$. If still $p_m^n(\vec{x}_{q_1}^1)\neq r_{q_1}^1$ then $p_m^n(\vec{x}_{q_1}^1)=p_m^n(\vec{x}_{q_2}^1)-d(\vec{x}_{q_1}^1,\vec{x}_{q_2}^1)$ for some $\vec{x}_{q_2}^1$. We continue until after finitely many steps we reach $\vec{x}_{q_n}^1$ such that $p_m^n(\vec{x}_{q_n}^1)=r_{q_n}^1$. However, observe that it follows from the series of triangle inequalities that $p_m^n(\vec{x}_j^1)=m_j^1=p_m^n(\vec{x}_{q_n}^1)-d(\vec{x}_j^1,\vec{x}_{q_n}^1)$ and we are done.
\end{enumerate}

Let us now prove that $\varepsilon_j^1-n/2^{k+3}\leq p_m^n(\vec{x}_j^1)$. Since we have $p_m^n(\vec{x}_j^1)=\min\{r_j^1,M_j^1\}$ it suffices to prove that $M_j^1\geq \varepsilon_j^1-n/2^{k+3}$. Again realize that $M_j^1=p_m^n(\vec{x}_p^1)+d(\vec{x}_j^1,\vec{x}_p^1)$ for some $\vec{x}_p^1$. There are again two cases:
\begin{enumerate}
\item There exists $\vec{x}_p^1\in S_1^n$ such that $M_j^1=p_m^n(\vec{x}_p^1)+d(\vec{x}_j^1,\vec{x}_p^1)$. Then since $\varepsilon_j^1\leq \varepsilon_p^1+d(\vec{b}_{\vec{x}_p^1},\vec{b}_{\vec{x}_j^1})$ we get from the inequalities above that $\varepsilon_j^1-n/2^{k+3}\leq M_j^1$.
\item If there is no such $\vec{x}_p^1\in S_1^n$ then as in item (2) above we can find $\vec{x}_p^1$ such that $M_j^1=p_m^n(\vec{x}_p^1)+d(\vec{x}_j^1,\vec{x}_p^1)$ and $p_m^n(\vec{x}_p^1)=r_p^1$. Then the verification is again analogous.

\end{enumerate}

\noindent {\bf Step 2.} Now we assume that $l>1$ and for all $m<l$ the claim has been proved.

If $p_m^n(\vec{x}_j^l)=r_j^l$ then it is clear. So we only have to prove that $m_j^l\leq \varepsilon_j^l+n/2^{k+l+2}$ and $\varepsilon_j^l-n/2^{k+l+2}\leq M_j^l$. We only prove the former, the latter is completely analogous.

We have $m_j^l=p_m^n(\vec{x}_p^l)-d(\vec{x}_j^l,\vec{x}_p^l)$ for some $\vec{x}_p^l$. As in Step 1 we find out that there (now) three possibilities (the verification that there precisely one of these three possibilities happens is similar to the verification that precisely one of those two possibilities in Step 1 happens).
\begin{enumerate}
\item There exists $\vec{x}_p^l\in (S_l\setminus \{u_k^{l-1}\})^n$ such that $m_j^l=p_m^n(\vec{x}_p^l)-d(\vec{x}_j^l,\vec{x}_p^l)$. Then it is analogous to the item (1) in Step 1.
\item There exists $\vec{x}_p^l$ such that $m_j^l=p_m^n(\vec{x}_p^l)-d(\vec{x}_j^l,\vec{x}_p^l)$ and $p_m^n(\vec{x}_p^l)=r_p^l$. This is analogous to the item (2) from Step 1.
\item There exists $\vec{x}_p^l$ such that $m_j^l=p_m^n(\vec{x}_p^l)-d(\vec{x}_j^l,\vec{x}_p^l)$ and $\vec{x}_p^l$ is an $n$-tuple obtained from $\vec{x}_j^l$ by replacing all occurences of $g$ by $u_k^{l-1}$, thus $\vec{x}_p^l$ is in fact equal to some $\vec{x}_q^{l-1}$ and $\varepsilon_j^l=\varepsilon_q^{l-1}$. By induction hypothesis we have that $p_m^n(\vec{x}_q^{l-1})\leq \varepsilon_j^l+(2n+1)/2^{l}$. Since $d(u_k^{l-1},g)=1/2^l$ we have that $d(\vec{x}_q^{l-1},\vec{x}_j^l)\geq 1/2^l\geq n/2^{k+l+2}$, thus $m_j^l=p_m^n(\vec{x}_q^{l-1})-d(\vec{x}_q^{l-1},\vec{x}_j^l)\leq \varepsilon_j^l+n/2^{k+l+2}$ as desired.

\end{enumerate}

\end{proof}

\begin{remark}
The previous proof can be slightly modified so that it proves Theorem \ref{basic}. We consider a language containing a symbol for rational metric and for every $n_i$, $i\leq m$, a symbol for rational $n_i$-ary function $p_{n_i}$. These functions are interpreted as distance functions from the desired closed sets $F_{n_i}$. Since there are only finitely many such rational functions they are all defined on all finite structures from $\Age$. The restrictions are the same, i.e. for any finite structure $A\in \Age$ we have for all $i\leq m$ that $\forall \vec{a},\vec{b}\in A^{n_i} (p_{n_i}(\vec{a})\leq p_{n_i}(\vec{a})+d(a_1,b_1)+\ldots +d(a_{n_i},b_{n_i}))$. The verification that such $\Age$ is a Fra\" iss\' e class is similar (only easier) as in Lemma \ref{isfraisse}. Similarly, the proof one-point extension property is similar, just easier, as in the proof of Proposition \ref{main_OPE}.
\end{remark}
\begin{observation}
The method used in the proof of Theorem \ref{main} to obtain countably many almost universal closed sets can be repeated in other instances. What we describe below is a general scheme. Note that we are very informal there and we refer to the proof of Theorem \ref{main} for an example with details.

Suppose we have a proof of universality and ultrahomogeneity of some metric structure using a Fra\" iss\' e limit of some class $\Age$ of structures in some language $\Lng$ consisting of rational metric and some other predicates or functions $p_1,\ldots,p_n$ with values in some fixed countable set. We may consider a new language consisting of the rational metric and predicates or functions $p_1^i,\ldots,p_n^i$ with values in the same fixed countable set for each $i\leq \Nat$. A structure $A$ belongs to this new class of structures $\tilde{\Age}$ if there is some $n_A$ (e.g. $|A|$) such that for all $i\leq n_A$ the functions (or predicates) $p_1^i,\ldots,p_n^i$ are defined on $A$ with the same restrictions for each $i$ as in $\Age$ for a single set of these functions (or predicates). The isomorphism and embedding relation between structures in $\tilde{\Age}$ is as in Definition \ref{defiso}. The verification that $\tilde{\Age}$ is a Fra\" iss\' e class is similar as in Lemma \ref{isfraisse}. The one-point extension property is also similar as in the proof of Proposition \ref{main_OPE}.
\end{observation}
\section{Universal and ultrahomogeneous closed subsets of $\Ur\times K$ and Lipschitz functions from $\Ur$ to $Z$}
In this section we consider a universal closed subset of $\Ur\times K$, where $K$ is an arbitrary fixed compact metric space, and a universal $L$-Lipschitz function from $\Ur$ to $Z$, where $L$ is an arbitrary fixed positive real number and $Z$ is an arbitrary fixed Polish metric space.
\begin{thm}
Let $K$ be an arbitrary compact metric space, $Z$ an arbitrary Polish metric space and $L\in \Rea^+$. Then the structure $(\Ur,C,F)$ is universal and ultrahomogeneous and unique with this property, where $C\subseteq \Ur\times K$ is a closed subset and $F:\Ur\rightarrow Z$ is an $L$-Lipschitz function.
\end{thm}
See the third and fourth example.
\begin{proof}
We split the proof into two parts. In order to increase transparency of the proof we separately prove that there is such a universal closed set $C\subseteq \Ur\times K$ and then that there is such a universal $L$-Lipschitz function $F:\Ur\rightarrow Z$. It will be a routine modification to prove that are ''simultaneously" universal and ultrahomogeneous. We will again use Fra\" iss\' e theory.\\

\noindent {\bf \underline{The closed set $C$}}.\\

Let $Q=\{q_n:q\in \Nat\}$ be an enumeration of a countable dense subset of $K$. We define the set $\mathcal{F}\subseteq (\Rat_0^+)^\Nat$ of all suitable functions. A function $f:\Nat\rightarrow \Rat_0^+$ belongs to $\mathcal{F}$ if there is a finite set $F\subseteq \Nat$ and non-negative rationals $r_i\geq 0$ for $i\in F$ such that $f(j)=\max\{0,\max\{r_i-d_K(q_j,q_i):i\in F\}\}$ and it is always the case that $f(i)=r_i$ for every $i\in F$; i.e. $f$ has the domain $\Nat$, however it is uniquely determined only by values on the finite set $F$. For $f\in \mathcal{F}$ we will denote such a finite set as $F_f$ (it is not unique, however there is unique such a set $F_f$ that is minimal in inclusion). Note that $\mathcal{F}$ is countable.

Let $p$ be an unary function with values in the set $\mathcal{F}$. Also, we again consider the binary rational function $d$ for metric. Let $\Lng$ be a language consisting precisely of these functions.

We now define the new class $\Age$ of finite structures of the language $\Lng$.
\begin{defin}
A finite structure $A$ for the language $\Lng$ of cardinality $k>0$ belongs to $\Age$ if the following conditions are satisfied.
\begin{enumerate}
\item $A$ is a finite rational metric space; i.e. it satisfies the same requirements as in the definition \ref{first_age}. 
\item The function $p$ is a total function, i.e. defined on all elements of $A$.

The interpretation of this functions is as follows: if $p(a)(n)=q>0$ then the distance between $(a,q_n)$ and $C$ is at least (in fact precisely) $q$; on the other hand, if $p(a)(n)=0$ then $(a,q_n)\in C$.
\item Here we describe the restriction that we must put on these functions.
\begin{itemize}
\item $\forall a,b\in A\forall n,m\in \Nat (p(a)(n)\leq p(b)(m)+d_K(q_n,q_m)+d(a,b)$
\end{itemize}

This requirement resembles the restriction from the proof of Theorem \ref{main}. The value $p(a)(n)$ determines a rational distance of $(a,q_n)$ in the sum metric from the set $C$. Thus in case that for example $p(a)(n)=0$, i.e. $(a,q_n)\in C$ and $p(b)(m)=q$, i.e. the distance in the sum metric of $(b,q_m)$ from $C$ is at least $q$, then necessarily the distance between $(a,q_n)$ and $(b,q_m)$ is at least $q$, i.e. $d(a,b)+d_K(q_n,q_m)\geq q$.

\end{enumerate}
\end{defin}
We must check that $\Age$ is again countable (up to isomorphism classes), satisfies hereditary, joint embedding and amalgamation property. The first two properties are clear. The verification of the third one is similar as in Theorem \ref{main}, we can just put the structures sufficiently far apart from each other. To check the amalgamation property, suppose we have structures $A,B,C$ such that $A$ is a substructure of both $B$ and $C$. We can again define $D$ with underlying set $A\coprod (B\setminus A)\coprod (C\setminus A)$ with metric extended so that $d(b,c)=\min\{d(b,a)+d(a,c):a\in A\}$ for $b\in B$ and $c\in C$, and $p^{D}(b)=p^{B}(b)$, resp. $p^{D}(c)=p^{C}(c)$, for $b\in B$, resp. $c\in C$ of course. Let us check that this works. Let $b\in B$, $c\in C$ and $n,m\in \Nat$. We check that $p(b)(n)\leq p(c)(m)+d(b,c)+d_K(q_n,q_m)$. Let $a\in A$ be such that $d(b,c)=d(b,a)+d(a,c)$. Then we have $p(b)(n)\leq p(a)(m)+d(b,a)+d_K(q_n,q_m)\leq p(c)(m)+d(a,c)+d(b,a)+d_K(q_n,q_m)=p(c)(m)+d(b,c)+d_K(q_n,q_m)$.

We again denote by $U$ the Fra\" iss\' e limit which is besides other things again a rational Urysohn space. We define the set $C\subseteq \Ur\times K$ in the completion $\Ur$ as follows: $(a,r)\in C \equiv \neg\exists (u,g)\in U\times Q \exists n\in \Nat (g=q_n\wedge d(a,u)+d_K(r,g)<p(u)(n))$. It is obviously closed.

Let us now state and prove the following useful claim that confirms that $p$ is really the distance function from the closed set.
\begin{claim}\label{pisdist}
For any $u\in U$ and $n\in \Nat$ we have $p(u)(n)=d((u,q_n),C)=q$.
\end{claim}
\begin{proof}[Proof of the claim.]
We prove that for an arbitrary $\varepsilon>0$ there exists $v\in U$ such that $d(u,v)<q+\varepsilon$ and $p(v)(n)=0$. It follows that $q\leq d(u,q_n),C)\leq q+\varepsilon$ for an arbitrary $\varepsilon>0$, thus $p(u)(n)=d((u,q_n),C)=q$. So let $\varepsilon>0$ be given. Let $d_\varepsilon\in \Rat^+$ be an arbitrary positive rational number so that $q\leq d_\varepsilon<q+\varepsilon$. Moreover, let $q_\varepsilon\in \Rat^+_0$ be an arbitrary nonnegative rational number smaller or equal to $d_q-q$. Let $F=\{i\in F_{p(u)}:p(u)(i)>d_\varepsilon\}$. We define $f\in \mathcal{F}$ such that $F_f=F$, and for $i\in F$ we set $f(i)=p(u)(i)-d_q+q_\varepsilon$. We define a one-point extension of $\{u\}$ as follows: the underlying set is $\{u,v\}$, we set $d(u,v)=d_\varepsilon$ and $p(v)=f$. We claim it belongs to $\Age$. Then by one-point extension property we can find such $v$ in $U$ and it is as desired: we need to prove that $p(v)(n)=0$. Suppose not, then there is $i\in F_f$ such that $p(v)(i)-d_K(q_i,q_n)>0$. However, since $p(u)(i)=p(v)(i)-q_\varepsilon+d_q$ we would have $p(u)(n)\geq p(u)(i)-d_K(q_i,q_n)=p(v)(i)-q_\varepsilon+d_q-d_K(q_i,q_n)>q$, a contradiction.

It remains to prove that $\{u,v\}\in \Age$. Let $n,m\in \Nat$ be given. We prove that $p(u)(n)\leq p(v)(m)+d(u,v)+d_K(q_n,q_m)$. If $p(u)(n)\leq d(u,v)$ then it is clear, so let us suppose that $p(u)(n)>d(u,v)$ and let $i\in F$ be such that $p(u)(n)=p(u)(i)-d_K(q_n,q_i)$. Then $p(v)(m)\geq p(v)(i)-d_K(q_m,q_i)\geq p(u)(i)-d(u,v)+q_\varepsilon-d_K(q_n,q_i)-d_K(q_n,q_m)$. It follows that $p(v)(m)+d(u,v)+d_K(q_n,q_m)\geq p(u)(i)-d_K(q_n,q_i)+q_\varepsilon= p(u)(n)+q_\varepsilon$.

Now we prove that also $p(v)(m)\leq p(u)(n)+d(u,v)+d_K(q_n,q_m)$. If $p(v)(m)=0$ then it is trivial, so let us suppose that $p(v)(m)>0$ and let $i\in F$ be such that $p(v)(m)=p(v)(i)-d_K(q_m,q_i)=p(u)(i)-d(u,v)+q_\varepsilon-d_K(q_m,q_i)$. Then $p(u)(n)\geq p(u)(i)-d_K(q_i,q_m)-d_K(q_n,q_m)$, thus $p(u)(n)+d(u,v)+d_K(q_n,q_m)\geq p(u)(i)-d_K(q_i,q_m)+d(u,v)\geq p(u)(i)-d(u,v)+q_\varepsilon-d_K(q_m,q_i)$ and we are done. Note that the last inequality follows from $d(u,v)\geq -d(u,v)+q_\varepsilon$ which is immediate from the definition of $d(u,v)$ and $q_\varepsilon$.
\end{proof}
\subsection{The one-point extension property for $(\Ur,C)$}
Let $\bar{\Age}$ be again the ``real" variant of $\Age$, i.e. a structure $A$ belongs to $\bar{\Age}$ if it is a finite metric space equipped with a closed subset $C_A$ of $A\times Z$, where $C_A$ need not to be finite. For each $n\in \Nat$ and $a\in A$ we denote by $p(a)(n)$ the distance of $(a,q_n)$ from $C_A$; $p(a)(n)$ in this case need not to be rational. The notions of embedding and isomorphism are obvious.

We again prove the one-point extension property for $\bar{\Age}$ which simplifies the proofs of universality, ultrahomogeneity and uniqueness of $(\Ur,C)$. By ``rational one-point extension property" we again mean the one-point extension property for structures from $\Age$.
\begin{prop}[The one-point extension property]\label{second_OPE}
Let $(A,C_A)$ be a finite substructure of $(\Ur,C)$ and let $(B,C_B)\in \bar{\Age}$ be a one-point extension, i.e. $|B|=|A|+1$ and there is an embedding $\phi:A\hookrightarrow B$. Then there exists an embedding $\psi:(B,C_B)\hookrightarrow (\Ur,C)$ such that $\mathrm{id}=\psi\circ \phi$.
\end{prop}
Before we provide the proof we again begin by showing how universality, ultrahomogeneity and uniqueness follow.
\begin{prop}
The Polish metric structure $(\Ur,C)$ is universal.
\end{prop}
\begin{proof}
Let $(X,d)$ be a Polish metric space and $B\subseteq X\times K$ a closed set. Let again $D=\{d_n:n\in \Nat\}\subseteq X$ be a countable dense set. We will find an isometric copy $D'$ of $D$ in $\Ur$ such that for any $d_n\in D$ and $q_m\in Q$ we have $d((d_n,q_m),B)=d((d'_n,q_m),C)$. This suffices. We can then extend the isometry, let us call it $\phi$, to the closure of $D$ which is the whole space $X$. Let $(x,r)\in X\times K$ be arbitrary. Assume at first that $(x,r)\notin B$. Let $\varepsilon=d((x,r),B)$. Then there exist $i,n\in \Nat$ such that $d((d_i,q_n),B)\geq 2\varepsilon/3$ and $d((d_i,q_n),(x,r))<\varepsilon/3$, thus $d((d'_i,q_n),C)\geq 2\varepsilon/3$, so $(\phi(x),r)\notin C$. On the other hand, assume that $(x,r)\in B$. Then there exists a sequence $(d_n,q_n)_n\subseteq D\times Q$ such that $(d_n,q_n)\to (x,r)$ and $(d(d_n,q_n),B)\to 0$. Thus also $(d'_n,q_n)\to (\phi(x),r)$ and since $d((d'_n,q_n),C)\to 0$ we have $d((\phi(x),r),C)=0$.

The construction of $D'$ is again just a series of applications of Proposition \ref{second_OPE}.
\end{proof}
\begin{claim}
The structure $(\Ur,C)$ is ultrahomogeneous and a unique structure having this kind of one-point extension property.
\end{claim}
Proofs are completely analogous to those in the first section.\\

\noindent\emph{Proof of Proposition \ref{second_OPE}.} 
Let us again at first treat the case when $A$ is empty and $B=\{b_1\}$. We just need to find some $a_1\in \Ur$ such that for every $n\in \Nat$ we have $p(a_1)(n)=p(b_1)(n)$. For every $l\in \Nat$ we define $f_l\in \mathcal{F}$ such that for every $n$ we shall have $|f_l(n)-p(b_1)(n)|<1/2^l$.

For every $n\in \Nat$ let $\varepsilon_n=d((b_k,q_n),C_B)$ ($=p(b_k)(n)$). Let $N\subseteq Q$ be a $1/2^{l+2}$-net in $K$, i.e. $\forall y\in K\exists x\in N (d_K(y,x)<1/2^{l+2})$. $N$ can be supposed to be finite since $K$ is totally bounded (this is the place where we need $K$ to be compact). Let $F$ be the set of indices of elements from $N$, i.e. $N=\{q_i:i\in F\}$. For every $i\in F$ let $\gamma_i^l\in \Rat_0^+$ be any non-negative rational number such that $0\leq \gamma_i^l-\varepsilon_i<1/2^{l+2}$. 

We define $f_l\in \mathcal{F}$. It suffices to define $f_l$ on a finite set $F$. Let  $F$ be equal to the set $\{i_1,\ldots,i_m\}$. WLOG we assume that $\gamma_{i_j}^l\geq \gamma_{i_l}^l$ for $j\leq l\leq m$.

We define $f_l$ inductively as follows: at step $1$ we set $f_l(i_1)=\gamma_{i_1}^l$. Suppose we are at step $n\leq m$. If $\eta_{i_n}^l=\max\{\gamma_i^l-d_K(q_i,q_{i_n}):i\in \{i_1,\ldots,i_{n-1}\}\}>\gamma_{i_n}^l$ then we set $f_l(i_n)=\eta_{i_n}^l$; otherwise, we set $f_l(i_n)=\gamma_{i_n}^l$. If we have finished then we have defined $f_l$ on $F$ ($=F_{f_l}$) which uniquely determines the values of $f_l$ on $\Nat$. We now check that for every $l,n\in \Nat$ we have $|f_l(n)-p(b_1)(n)|<1/2^l$. Let $n\in \Nat$ be arbitrary. There exists $i\in F$ such that $d_K(q_i,q_n)<1/2^{l+2}$. Since it follows $|\varepsilon_i-\varepsilon_n|<1/2^{l+2}$ and $|p(u_1^l)(n)-p(u_1^l)(i)|<1/2^{l+2}$ (the functions $q_i\rightarrow \varepsilon_i$ and $q_i\rightarrow p(u_1^l)(i)$ are $1$-Lipschitz) it suffices to check that for any $n\in F_{f_l}$ we have $|p(u_1^l)(n)-\varepsilon_n|\leq 1/2^{l+1}$. For $n\in F_{f_l}$ we either have that $p(u_1^l)(n)=\gamma_n^l$ or that $p(u_1^l)(n)=\eta_n^l$. If the former case holds then it is clear from the choose of $\gamma_n^l$. If $p(u_1^l)(n)=\eta_n^l$ then from the definition of $\eta_n^l$ we have $\eta_n^l>\gamma_n^l$ and there exists $i\in F$ such that $p(u_1^l)(i)=\gamma_i^l$ and $p(u_1^l)(n)=\eta_n^l=p(u_1^l)(i)-d_K(q_i,q_n)$. Since $\eta_n^l>\gamma_n^l\geq \varepsilon_n-1/2^{l+1}$ it suffices to check that $\eta_n^l\leq \varepsilon_n+1/2^{l+1}$. However since $\varepsilon_i\leq \varepsilon_n+d_K(q_i,q_n)$ and $|\gamma_i^l-\varepsilon_i|\leq 1/2^{l+1}$ this follows.

Now we use the rational one-point extension property to define a sequence $(u_1^j)_j\subseteq U$ such that for every $j\in \Nat$ we have $p(u_1^j)=f_j$ and $d_\Ur(u_1^j,u_1^{j+1})=1/2^{j+1}$. It is straightforward to check that this is possible and since for every $j,n\in \Nat$ we have $|p(u_1^l)(n)-p(b_1)(n)|<1/2^l$ it follows that $p(a_1)(n)=p(b_1)(n)$, for every $n\in \Nat$, where $a_1=\lim_l u_1^l$.\\
 
We now assume that $A$ is non-empty. Let us enumerate $A$ as $\{a_1,\ldots,a_{k-1}\}$ and $B$ as $\{b_1,\ldots,b_k\}$ in such a way that for every $i< k$ we have $\phi(a_i)=b_i$. We shall find a new point $a_k\in \Ur$ such that the structures $A\cup\{a_k\}$ and $B$ will be isomorphic. For every $i< k$ let $(u_i^j)_j\subseteq U$ be a sequence from the rational space $U$ converging to $a_i$ such that $d_\Ur(u_i^ j,a_i)<1/2^i$, for every $n\in \Nat$ $|p(a_i)(n)-p(u_i^j)(n)|<1/2^{j+1}$ and for any pair $j<l$ we have $d_\Ur(u_i^j,a_i)>d_\Ur(u_i^l,a_i)$. We shall find a new sequence from $U$ converging to the desired point $a_k$. This is done by induction.

Consider a structure $S_1=\{u_1^{k+3},\ldots,u_{k-1}^{k+3}\}$ such that for every $i< k$ and $n\in \Nat$ we have $p(u_i^{k+3})(n)=d((u_i^{k+3},q_n),C)$. Thus $S_1\in \Age$. We use Lemma \ref{metriclem} to define a metric one-point extension $M_1=\{u_1^{k+3},\ldots,u_{k-1}^{k+3},g\}$ such that for all $i< k$ we have $d(b_i,b_k)\leq d_\Ur(u_i^{k+3},g)\leq d(b_i,b_k)+1/2$. We define a structure $V_1$ such that $M_1$ is its underlying (rational) metric space. We need to define $p$ on $g$. This will be similar to the definition of $p$ on $u_1^l$'s (from case when $A$ was empty) but more complicated.

For every $n\in \Nat$ let $\varepsilon_n=d((b_k,q_n),C_B)$ ($=p(b_k)(n)$). 
Let $N\subseteq Q$ be a $1/2^3$-net in $K$, i.e. $\forall y\in K\exists x\in N (d_K(y,x)<1/2^3)$. $N$ can be supposed to be finite since $K$ is totally bounded. Let $F'$ be the set of indices of elements from $N$, i.e. $N=\{q_i:i\in F'\}$. We set $F=F'\cup \bigcup _{i< k} F_{p(u_i^{k+3})}$. For every $i\in F$ let $\delta_i^1\in \Rat_0^+$ be any non-negative rational number such that $0\leq \delta_i^1-\varepsilon_i<1/2^2$. Also, we define $m_i^1$ to be $\max\{p(u_j^{k+3})(i)-d_\Ur(g,u_j^{k+3}):j< k\}$ and $M_i^1$ to be  $\min\{p(u_j^{k+3})(i)+d_\Ur(g,u_j^{k+3}):j< k\}$. For every $i\in F$ if $m_i^1\leq \delta_i^1\leq M_i^1$ then we set $\gamma_i^1=\delta_i^1$. If $\delta_i^1<m_i^1$, resp. $M_i^1<\delta_i^1$ then we set $\gamma_i^1=m_i^1$, resp. $\gamma_i^1=M_i^1$. 

We define $f\in \mathcal{F}$. It suffices to define $f$ on a finite set $F$. Let  $F$ be equal to the set $\{i_1,\ldots,i_m\}$. WLOG we assume that $\gamma_{i_j}^1\geq \gamma_{i_l}^1$ for $j\leq l\leq m$.

We define $f$ inductively as follows: at step $1$ we set $f(i_1)=\gamma_{i_1}^1$. Suppose we are at step $n\leq m$. If $\eta_{i_n}^1=\max\{\gamma_i^1-d_K(q_i,q_{i_n}):i\in \{i_1,\ldots,i_{n-1}\}\}>\gamma_{i_n}^1$ then we set $f(i_n)=\eta_{i_n}^1$; otherwise, we set $f(i_n)=\gamma_{i_n}^1$. If we have finished then we have defined $f$ on $F$ ($=F_f$) which uniquely determines the values of $f$ on $\Nat$. 

We now put $p(g)=f$. It is straightforward to check it is consistent. We defined an extension $V_1\in \Age$ of $S_1$ and thus there is some $u_k^1\in U$ playing the role of $g$.\\

Suppose we have already constructed $u_k^1,\ldots,u_k^{l-1}\subseteq U$ such that $d_\Ur(u_k^i,u_k^{i+1})=1/2^{i+1}$ for any $i<l-1$. Consider a structure $S_l=\{u_1^{k+l+2},\ldots,u_{k-1}^{k+l+2},u_k^{l-1}\}$ with $p(u)(n)=d((u,q_n),C)$ for every $u\in S_l$ and $n\in \Nat$. We use Lemma \ref{metriclem} to obtain a metric extension $M_l=\{u_1^{k+l+2},\ldots,u_{k-1}^{k+l+2},u_k^{l-1},g\}$ such that such that for all $i< k$ we have $d(b_i,b_k)\leq d_\Ur(u_i^{k+l+2},g)\leq d(b_i,b_k)+1/2^l$. We need to define $p$ on $g$. This is done in the same way as in the first induction step: For every $n\in \Nat$ let $\varepsilon_n=d((b_k,q_n),C_B)$ ($=p(b_k)(n)$). 
Let $N\subseteq Q$ be a $1/2^{l+2}$-net in $K$, i.e. $\forall y\in K\exists x\in N (d_K(y,x)<1/2^{l+2})$. $N$ can be supposed to be finite since $K$ is totally bounded. Let $F'$ be the set of indices of elements from $N$, i.e. $N=\{q_i:i\in F'\}$. We set $F=F'\cup \bigcup _{i< k} F_{p(u_i^{k+l+2})}$. For every $i\in F$ let $\delta_i^l\in \Rat_0^+$ be any non-negative rational number such that $0\leq \delta_i^l-\varepsilon_i<1/2^{l+2}$. Also, we define $m_i^l$ to be $\max\{p(u)(i)-d_\Ur(g,u):u\in \{u_j^{k+l+2}:j< k\}\cup\{g\}\}$ and $M_i^l$ to be  $\min\{p(u)(i)+d_\Ur(g,u):u\in \{u_j^{k+l+2}:j< k\}\cup\{g\}\}$. For every $i\in F$ if $m_i^l\leq \delta_i^l\leq M_i^l$ then we set $\gamma_i^l=\delta_i^l$. If $\delta_i^l<m_i^l$, resp. $M_i^l<\delta_i^l$ then we set $\gamma_i^l=m_i^l$, resp. $\gamma_i^l=M_i^l$. 

We define $f\in \mathcal{F}$. It suffices to define $f$ on a finite set $F$. Let  $F$ be equal to the set $\{i_1,\ldots,i_m\}$. WLOG we assume that $\gamma_{i_j}^l\geq \gamma_{i_l}^l$ for $j\leq l\leq m$.

We define $f$ inductively as follows: at step $1$ we set $f(i_1)=\gamma_{i_1}^l$. Suppose we are at step $n\leq m$. If $\eta_{i_n}^l=\max\{\gamma_i^l-d_K(q_i,q_{i_n}):i\in \{i_1,\ldots,i_{n-1}\}\}>\gamma_{i_n}^l$ then we set $f(i_n)=\eta_{i_n}^l$; otherwise, we set $f(i_n)=\gamma_{i_n}^l$. If we have finished then we have defined $f$ on $F$ ($=F_f$) which uniquely determines the values of $f$ on $\Nat$. 

We now put $p(g)=f$. It is straightforward to check it is consistent. We defined an extension $V_l\in \Age$ of $S_l$ and thus there is some $u_k^l\in U$ playing the role of $g$.

Assume the induction is finished. We have produced a sequence $(u_k^j)_j\subseteq U$ such that for every $i\in \Nat$ we have $d_\Ur(u_k^i,u_k^{i+1})=1/2^{i+1}$ thus the sequence is Cauchy and we denote $a_k$ its limit point. It immediately follows from the construction that $d_\Ur(a_i,a_k)=d(b_i,b_k)$ for every $i< k$. It remains to check that for every $y\in K$ $d((a_k,y),C)=d((b_k,y),C_B)$. It obviously suffices to check that for every $n\in \Nat$ $d((a_k,q_n),C)=d((b_k,q_n),C_B)$.

\begin{claim}
Let $l,n\in \Nat$ be arbitrary. Then $|p(u_k^l)(n)-\varepsilon_n|\leq 1/2^l$.
\end{claim}
Once the claim is proved the previous assertion is clear so it remains to prove the claim.\\

\noindent \emph{Proof of the claim.} As in the proof of the analogous Claim \ref{correct_estim_main} we prove it by induction on $l$.\\

\noindent {\bf Step 1.}\\
Suppose $l=1$ (in some places where it may be confusing we shall still use the symbol $l$ eventhough it is equal to $1$ in Step 1). Let $n\in \Nat$ be arbitrary. There exists $i\in F$ such that $d_K(q_i,q_n)<1/2^{l+2}=1/2^2$. Since it follows $|\varepsilon_i-\varepsilon_n|<1/2^{l+2}$ and $|p(u_k^1)(n)-p(u_k^1)(i)|<1/2^{l+2}$ (the functions $q_i\rightarrow \varepsilon_i$ and $q_i\rightarrow p(u_k^1)(i)$ are $1$-Lipschitz) it suffices to check that for any $n\in F$ we have $|p(u_k^1)(n)-\varepsilon_n|\leq 1/2^{l+1}$.

From the definition of $p(u_k^1)$ above we have two cases:
\begin{itemize}
\item  $p(u_k^1)(n)=\gamma_n^1$. This splits into three subcases:
\begin{enumerate}
\item $p(u_k^1)(n)=\delta_n^1$. However we defined that $0\leq \delta_n^1-\varepsilon_n\leq 1/2^{l+1}=1/2^2$ so we are done.
\item $p(u_k^1)(n)=m_n^1$. In that case $m_n^1> \delta_n^1$ and we must check that $m_n^1\leq \varepsilon_n+1/2^{l+1}$.

From the definition there is some $i< k$ such that $m_n^1=p(u_i^{k+l+2})(n)-d(u_i^{k+l+2},u_k^1)$. However, from the assumption we have $|p(u_i^{k+l+2})(n)-p(a_i)(n)|<1/2^{k+l+2}$ and recall that $d(b_i,b_k)+1/(k\cdot 2^{l+1})\leq d(u_i^{k+l+2},u_k^1)$. Since $p(b_i)(n)\leq \varepsilon_n+d(b_i,b_k)$ (recall that $p(b_i)(n)=p(a_i)(n)$ and $\varepsilon_n=p(b_k)(n)$), putting these three inequalities together the inequality $m_n^1\leq \varepsilon_n+1/2^{l+1}$ follows.
\item $p(u_k^1)(n)=M_n^1$. In that case $M_n^1< \delta_n^1$ and we must check that $M_n^1\geq \varepsilon_n^1-1/2^{l+1}$. From the definition there is some $i< k$ such that $M_n^1=p(u_i^{k+l+2})(n)+d(u_i^{k+l+2},u_k^1)$. We again use the inequalities from the previous item, i.e. $|p(u_i^{k+l+2})(n)-p(a_i)(n)|<1/2^{k+l+2}$ and $d(b_i,b_k)+1/(k\cdot 2^{l+1})\leq d(u_i^{k+l+2},u_k^1)$. Moreover, since $\varepsilon_n(=p(b_k)(n))\leq p(b_i)(n)+d(b_i,b_k)$, putting these three inequalities together the inequality $M_n^1\geq \varepsilon_n-1/2^{l+1}$ follows.
\end{enumerate}
\item $p(u_k^1)(n)=\eta_n^1$. Then it follows from the definition that there exists some $i\in F$ such that $p(u_k^1)(i)=\gamma_i^1$ and $\eta_n^1=\gamma_i^1-d_K(q_n,q_i)>\delta_n^1$. Since we already know from the previous item that $\delta_n^1\geq \varepsilon_n-1/2^{l+1}$ and we know that $\eta_n^1>\delta_n^1$ we have that $\eta_n^1>\varepsilon_n-1/2^{l+1}$. Thus it suffices to check that $\eta_n^1\leq \varepsilon_n+1/2^{l+1}$. We again have three subcases:
\begin{enumerate}
\item $\gamma_i^1=\delta_i^1$. We have that $p(b_k)(i)(=\varepsilon_i)\leq p(b_k)(n)(=\varepsilon_n)+d_K(q_i,q_n)$. Since we know from the previous item that $\delta_i^1\leq \varepsilon_i+1/2^{l+1}$ and since $\eta_n^1=\delta_i^1-d_K(q_n,q_i)$ we get that $\eta_n^1\leq \varepsilon_n+1/2^{l+1}$.
\item $\gamma_i^1=M_i^1$. In that case we have that $M_i^1<\delta_i^1$ thus the inequality $\eta_n^1\leq \varepsilon_n+1/2^{l+1}$ follows from (1) immediately above.
\item $\gamma_i^1=m_i^1$. In that case there is some $j< k$ such that $\gamma_i^1=m_i^1=p(u_j^{k+l+2})(i)+d(u_j^{k+l+2},u_k^1)$. Since $p(b_j)(i)\leq \varepsilon_n(=p(b_k)(n))+d_K(q_i,q_n)+d(b_i,b_k)$, using the inequalities from (2) and (3) from the previous item we get that $\eta_n^1\leq \varepsilon_n+1/2^{l+1}$.

\end{enumerate}
\end{itemize}
{\bf Step 2.}\\
Now we assume that $l>1$ and for all $i<l$ the claim has been proved. Let again $n\in \Nat$ be arbitrary. Then there exists $i\in F$ such that $d_K(q_i,q_n)<1/2^{l+2}$. Thus it again suffices to check that for any $n\in F$ we have $|p(u_k^1)(n)-\varepsilon_n|\leq 1/2^{l+1}$. There are again two cases: either $p(u_k^l)(n)=\gamma_n^l$ or $p(u_k^l)(n)=\eta_n^l$. Both of them are treated similarly as in Step 1; let us illustrate it only on the former. We again have three subcases:
\begin{enumerate}
\item $p(u_k^l)(n)=\delta_n^l$. However we defined that $0\leq \delta_n^l-\varepsilon_n\leq 1/2^{l+1}$ so we are done.
\item $p(u_k^l)(n)=m_n^l$. In that case $m_n^l> \delta_n^l$ and we must check that $m_n^l\leq \varepsilon_n+1/2^{l+1}$.

From the definition there is some $u\in \{u_i^{k+l+2}:i< k\}\cup\{u_k^{l-1}\}$ such that $m_n^1=p(u)(n)-d(u,u_k^l)$. If $u\in \{u_i^{k+l+2}:i< k\}$ then the proof is completely analogous to the corresponding item in Step 1. So we assume that $u=u_k^{l-1}$. However, we have from the induction hypothesis that $|p(u_k^{l-1})(n)-\varepsilon_n|<1/2^l$ and since $d(u_k^{l-1},u_k^l)=1/2^l$ we have that $m_n^l\leq \varepsilon_n+1/2^{l+1}$.

\item $p(u_k^l)(n)=M_n^l$. In that case $M_n^l< \delta_n^l$ and we must check that $M_n^l\geq \varepsilon_n^1-1/2^{l+1}$. From the definition there is some $u\in \{u_i^{k+l+2}:i< k\}\cup\{u_k^{l-1}\}$ such that $M_n^l=p(u)(n)-d(u,u_k^l)$. Again as in (2) above, if $u\in \{u_i^{k+l+2}:i< k\}$ then the proof is completely analogous to the corresponding item in Step 1, so we assume that $u=u_k^{l-1}$. However, we again use the induction hypothesis that $|p(u_k^{l-1})(n)-\varepsilon_n|<1/2^l$ and since $d(u_k^{l-1},u_k^l)=1/2^l$ we have that $M_n^l\geq \varepsilon_n^1-1/2^{l+1}$.
\end{enumerate}
This finishes the proof of the claim and also of Proposition \ref{second_OPE}.\\
\bigskip

\noindent {\bf \underline{The Lipschitz function $F$}}\\

Let a Lipschitz constant $L\in \Rea^+$ be fixed. Let $Q=\{q_n:n\in \Nat\}$ be an enumeration of some fixed countable dense subset $Q$ of the Polish metric space $Z$.Let $p$ be an unary function with values in $\Nat$ and $d$ again a binary rational function. Let $\Lng$ be a language consisting of these functions.

We again define the (new) class $\Age$ of structures in the language $\Lng$ and then prove it satisfies the required properties of the Fra\" iss\' e theory.
\begin{defin}
A finite structure $A$ for the language $\Lng$ of cardinality $k$ belongs to $\Age$ if the following conditions are satisified
\begin{enumerate}
\item $A$ is again a finite rational metric space, i.e. it satisfies the same requirements as in definitions before. We will again interpret $d$ as a metric.
\item The function $p$ is a total function.

The intended interpretation of this function is that the value $p(a)$ determines the value of the universal continuous function $F$ at $a$ as follows: $F(a)=q_{p(a)}$. 
\item Here we put the restrictions on these structures which is just the demand that the desired function $F$ is $L$-Lipschitz. For every $a$ and $b$ from $A$ $d_Z(q_{p(a)},q_{p(b)})\leq L\cdot d(a,b)$. 
\end{enumerate}
\end{defin}

Now we verify that $\Age$ satisfies all properties needed to have a Fra\"i ss\' e limit. The countability and hereditary property are clear. To check joint embedding property, consider two structures $A$ and $B$. Consider again $m_A$ defined as $\max\{d(a,b):a,b\in A\}$, $m_B$ defined analogously for $B$ and moreover, $m_F=\max\{L\cdot d_Z(q_{p(a)},q_{p(b)}):a\in A,b\in B\}$. Set $m=\max\{m_A,m_B,m_F\}$ and define the metric on $A\coprod B$ as follows: for $a\in A$, $b\in B$, $d(a,b)=2m$. This again works.

Finally, we need to check the amalgamation property. So let $A,B,C\in \Age$ be structures and we assume that $A$ is a substructure of both $B$ and $C$. We set $D=A\coprod (B\setminus A)\coprod (C\setminus A)$. The metric is again extended in the standard way, i.e. for $b\in B$ and $c\in C$ we set $d(b,c)=\min\{d(b,a)+d(a,c):a\in A\}$.

We need to check that for any $b\in B$ and $c\in C$ we still have $d_Z(q_{p(b)},q_{p(c)})\leq L\cdot d(b,c)$. Let $a\in A$ be such that $d(b,c)=d(b,a)+d(a,c)$. We have $d_Z(q_{p(b)},q_{p(c)})\leq d_Z(q_{p(b)},q_{p(a)})+d_Z(q_{p(a)},q_{p(c)})\leq L\cdot d(b,a)+L\cdot d(a,c)=L\cdot d(b,c)$.\\

We again denote the Fra\" iss\' e limit as $U$. We define a function $\tilde{F}$ on $U$ to $Z$ as follows: $\tilde{F}(u)=q_{p(u)}$. It follows from our construction that $\tilde{F}$ is $L$-Lipschitz, thus we may extend $\tilde{F}$ to the completion $\Ur$; we denote $F$ this unique $L$-Lipschitz extension and claim that this is the desired universal $L$-Lipschitz function to the Polish metric space $Z$.
\subsection{The one-point extension property for $(\Ur,F)$}
We again prove a particular version of one-point extension property. The method how to use it to derive the universality, ultrahomogeneity and uniqueness is the same as before. By $\bar{\Age}$ we denote the class of all finite metric spaces equipped with an $L$-Lipschitz function into $Z$. Recall that we have the rational one-point extension property concerning structures from $\Age$.
\begin{prop}[One-point extension property]
Let $A$ be a finite substructure of $(\Ur,F)$ and let $B\in\bar{\Age}$ be such that $|B|=|A|+1$ and there is an embedding $\phi$ of $A$ into $B$. Then there exists an embedding $\psi$ of $B$ into $(\Ur,F)$ such that $\mathrm{id}=\psi \circ \phi$.
\end{prop}
\noindent\emph{Proof of the proposition.}
We again start with the case when $A$ is empty and $B=\{b_1\}$. We just need to find some $a_1\in \Ur$ such that $F(a_1)=F(b_1)$. Choose some sequence $(f_1^l)_l\subseteq \Nat$ such that for every $n\in \Nat$ $d_Z(q_{f_1^n},q_{f_1^{n+1}})\leq L/2^{n+1}$ and $q_{f_1^l}\to F(b_1)$. Using the rational one-point extension property we find a sequence $(u_1^j)_j\subseteq$ such that for every $n\in \Nat$ we have $p(u_1^n)=f_1^n$ and $d_\Ur(u_1^n,u_1^{n+1})=1/2^{n+1}$. This is possible and we have that $F(a_1)=F(b_1)$ where $a_1=\lim_n u_1^n$.\\

We now assume that $A$ is non-empty. Let us enumerate $A$ as $\{a_1,\ldots,a_{k-1}\}$ and $B$ as $\{b_1,\ldots,b_k\}$ so that the embedding $\phi$ of $A$ into $B$ sends $a_i$ to $b_i$ for every $i< k$. We shall find a new point $a_k\in \Ur$ and define an embedding $\psi:B\hookrightarrow (\Ur,F)$ sending $b_i$ to $a_i$ for every $i\leq k$. We will find a Cauchy sequence of elements from $U$ such that the limit will be this desired point $a_k$. For each $l<k$ let us choose a converging sequence $(u_l^j)_j\subseteq U$ of elements from the Fra\" iss\' e limit such that $\lim_j u_l^j=a_l$, $d_\Ur(u_l^j,a_l)<1/2^j$, for $i<j$ we have $d_\Ur(u_l^j,a)<d_\Ur(u_l^i,a)$, and moreover for every natural numbers $i>j$ we have $d_Z(F(u_l^j),F(u_l^i))<L/(k\cdot 2^{j+2})$. For every $l< k$ and $i\in\Nat$ let $f_l^i\in \Nat$ be such that $F(u_l^i)=q_{f_l^i}$.

Now, let us a choose a sequence $(f_k^j)_j\subseteq \Nat$ of natural numbers such that $\forall j\in \Nat\forall i>j(d_Z(q_{f_k^j},q_{f_k^i})<L/(k\cdot 2^{j+2}))$ and $q_{f_k^j}\to F(b_k)$. Consider a structure $S_1=\{u_1^{k+3},\ldots,u_{k-1}^{k+3}\}$. For every $a\in S_1$ we set $p(a)=n$ iff $F(a)=q_n$, thus $S_1\in \Age$. We use Lemma \ref{metriclem} to find a metric extension $M_1=\{u_1^{k+3},\ldots,u_{k-1}^{k+3},g\}$ such that for all $i< k$ we have $d(b_i,b_k)+1/(k\cdot 2^2)\leq d_\Ur(u_i^{k+3},g)\leq d(b_i,b_k)+1/2$. We extend $M_1$ into a structure $V_1$ from $\Age$. We just need to define $p$ on $g$. We set $p(g)=f_k^1$. To check that this is consistent we need to verify $d_Z(q_{f_k^1},q_{f_j^{k+3}})\leq L\cdot d(u_j^{k+3},g)$ for all $j<k$. However, since $d_Z(F(b_i),f_j^{k+3})\leq L/(k\cdot 2^{k+5}))$ and $d(b_j,b_k)+1/(k\cdot 2^2)\leq d_\Ur(u_j^{k+3},g)$ it follows that $$d_Z(q_{f_k^1},q_{f_j^{k+3}})\leq d_Z(F(b_j),F(b_k))+L/(k\cdot 2^{k+5})+L/(k\cdot 2^3)\leq L\cdot d(b_j,b_k)$$ $$+L/(k\cdot 2^2)) \leq L\cdot d(u_j^{k+3},g)-L/(k\cdot 2^2) +L/(k\cdot 2^2))\leq L\cdot d(u_j^{k+3},g)$$
Thus $V_1\in \Age$ and there is some $u_k^1\in U$ playing the role of $g$.\\

Suppose we have already constructed $u_k^1,\ldots,u_k^{l-1}\subseteq U$. We consider a structure $S_l=\{u_1^{k+l+2},\ldots,u_{k-1}^{k+l+2},u_k^{l-1}\}$ with an obvious definition of $p$ on elements of $S_l$. We use again Lemma \ref{metriclem} to obtain a metric extension $M_l=\{u_1^{k+l+2},\ldots,u_{k-1}^{k+l+2},u_k^{l-1},g\}$ such that for all $i, k$ we have $d(b_i,b_k)+1/(k\cdot 2^{l+1})\leq d_\Ur(u_i^{k+l+2},g)\leq d(b_i,b_k)+1/2^l$. We need to define $p$ on $g$; we set $p(g)=f_k^l$. The verification that it is consistent is the same as above. So we obtain some $u_k^l\in U$ playing the role of $g$. This finishes the induction and the proof.
\end{proof}

\end{document}